%% file: article-decomp.tex
\pdfoutput=1
\documentclass{amsart}

\input{packages}
\input{configuration}
\input{macros}


\title{Decompositions of moduli spaces of vector bundles and graph potentials}

\author{Pieter Belmans}
\address{Department of Mathematics, Universit\'e de Luxembourg, 6, avenue de la Fonte, L-4364 Esch-sur-Alzette, Luxembourg}
\author{Sergey Galkin}
\address{
  PUC-Rio, Departamento de Matem\'atica, Rua Marqu\^es de S\~ao Vicente 225, G\'avea, Rio de Janeiro \\
  (on leave from HSE University, Russian Federation)
}
\author{Swarnava Mukhopadhyay}
\address{School of Mathematics, Tata Institute of Fundamental Research, 1 Homi Bhabha Road, Navy Nagar, Colaba, Mumbai 400005}

\email{pieter.belmans@uni.lu,arxiv-gp-decomp@galkin.org.ru,swarnava@math.tifr.res.in}

\begin{document}

\begin{abstract}
  We propose a conjectural semiorthogonal decomposition for the derived category of
  the moduli space of stable rank 2 bundles with fixed determinant of odd degree,
  independently formulated by Narasimhan.
  We discuss some evidence for,
  and furthermore propose semiorthogonal decompositions with additional structure.

  We also discuss two other decompositions.
  One is a decomposition of this moduli space in the Grothendieck ring of varieties,
  which relates to various known motivic decompositions.
  The other is the critical value decomposition of a candidate mirror Landau--Ginzburg model given by graph potentials,
  which in turn is related under mirror symmetry to Mu\~noz's decomposition of quantum cohomology.
  This corresponds to an orthogonal decomposition of the Fukaya category.
  We will explain how these decompositions can be seen as evidence for the conjectural semiorthogonal decomposition.
\end{abstract}

\maketitle

\tableofcontents

\input{decomp-intro}
\input{decomp-sod}

\input{decomp-qh-gp}

\input{decomp-kzero}
\input{decomp-critical}

\renewcommand*{\bibfont}{\small}
\printbibliography
\end{document}

%% file: packages.tex
\usepackage{libertine}
\usepackage[T1]{fontenc}
\usepackage[utf8]{inputenc}      % for texlive-2016 and older
\usepackage[libertine]{newtxmath}
\usepackage[scaled=0.83]{beramono}
\usepackage{eucal}
\usepackage[babel]{microtype}
\usepackage[russian,main=american]{babel}
\usepackage[main=american]{babel}
\usepackage[strict]{csquotes}
\usepackage{alphabeta}              % in LuaLaTeX use {unicode-math}
\usepackage[fleqn,leqno]{mathtools} % loads amsmath, makes \colonequals from :=
\usepackage[fleqn,leqno]{amsmath}   % not needed with amsart, needed by article
\usepackage[acronym,toc,nomain]{glossaries}  % for a list of acronyms
\usepackage{hyperref}               % options in hyperref.cfg & personal.sty
\usepackage[capitalise]{cleveref}   % _after_ hyperref & amsmath. Use: in \Cref{t:1}
\usepackage[dvipsnames]{xcolor}     % used by \hypersetup
\usepackage[ocgcolorlinks]{ocgx2}[2017/03/30]	% linebreak links and uncolor in print

   \let\digamma\relax                   % in front of amsfonts
\usepackage{amsfonts}
\usepackage{amsthm}
\usepackage{booktabs}
\usepackage{changepage}
\usepackage{diagbox}
\usepackage{dynkin-diagrams}
\usepackage{enumitem}
\usepackage{parskip}
\usepackage{pgfplots}
\pgfplotsset{compat=1.16}     % less compatibility
\usepackage{rotating}
\usepackage{siunitx}
\usepackage{subcaption}
\usepackage{thmtools}
\usepackage{tikz}
\usepackage{tikz-cd}
\usetikzlibrary{babel}
\usepackage[colorinlistoftodos, textsize = footnotesize]{todonotes}
\usepackage{xparse}
\usepackage{xspace}
\usetikzlibrary{calc,tqft,decorations.markings}

\usepackage[backend=biber,maxbibnames=99,doi=false]{biblatex}

%% file: configuration.tex
\newtoggle{nome}
\toggletrue{nome}
\mathtoolsset{centercolon}

\makeatletter
\def\paragraph{\vspace{6pt}\@startsection{paragraph}{4}%
  \z@\z@{-\fontdimen2\font}%
  {\normalfont\bfseries}}
\makeatother

\newcounter{todocounter}
\DeclareDocumentCommand\addreference{g}{\stepcounter{todocounter}\todo[color = blue!30]{\thetodocounter. Add reference\IfNoValueF{#1}{: #1}}\xspace}
\DeclareDocumentCommand\checkthis{g}{\stepcounter{todocounter}\todo[color = red!50]{\thetodocounter. Check this\IfNoValueF{#1}{: #1}}\xspace}
\DeclareDocumentCommand\fixthis{g}{\stepcounter{todocounter}\todo[color = orange!50]{\thetodocounter. Fix this\IfNoValueF{#1}{: #1}}\xspace}
\DeclareDocumentCommand\expand{g}{\stepcounter{todocounter}\todo[color = green!50]{\thetodocounter. Expand\IfNoValueF{#1}{: #1}}\xspace}

\tikzset{every loop/.style={min distance=1cm, looseness=30}}
\tikzset{vertex/.style={circle, draw, black, minimum size=6pt, inner sep=0pt}}
\tikzset{half-vertex/.style={semicircle, draw, fill, black, minimum size=3pt, inner sep=0pt, yshift=1.5pt}}

\tikzset{middlearrow/.style={decoration={markings, mark= at position 0.5 with {\arrow{#1}},}, postaction={decorate}}}

\declaretheoremstyle[
  spaceabove = 3pt,
  spacebelow = 3pt,
  bodyfont = \itshape,
]{first}
\declaretheorem[numberwithin=section, style=first]{theorem}
\declaretheorem[sibling=theorem, style=first]{conjecture}
\declaretheorem[sibling=theorem, style=first]{corollary}
\declaretheorem[sibling=theorem, style=first]{lemma}
\declaretheorem[sibling=theorem, style=first]{proposition}

\declaretheorem[numberwithin=section, style=first, title=Theorem]{alphatheorem}
\declaretheorem[sibling=alphatheorem, style=first, title=Conjecture]{alphaconjecture}

\declaretheoremstyle[
  spaceabove = 0pt,
  spacebelow = 0pt,
]{second}
\theoremstyle{second}
\declaretheorem[style=second, sibling=theorem]{construction}
\declaretheorem[style=second, sibling=theorem]{definition}
\declaretheorem[style=second, sibling=theorem]{example}

\declaretheorem[style=second, sibling=theorem]{remark}

\crefname{conjecture}{Conjecture}{Conjectures}
\crefname{construction}{Construction}{Constructions}

\crefname{alphatheorem}{Theorem}{Theorems}
\crefname{alphaconjecture}{Conjecture}{Conjectures}
\crefname{alphacorollary}{Corollary}{Corollaries}
\crefname{alphaproposition}{Proposition}{Propositions}

\crefname{figure}{Figure}{Figures}

\makeatletter
\def\gitfootnote{\gdef\@thefnmark{}\@footnotetext}
\makeatother

\makeatletter
\DeclareRobustCommand\widecheck[1]{{\mathpalette\@widecheck{#1}}}
\def\@widecheck#1#2{%
    \setbox\z@\hbox{\m@th$#1#2$}%
    \setbox\tw@\hbox{\m@th$#1%
       \widehat{%
          \vrule\@width\z@\@height\ht\z@
          \vrule\@height\z@\@width\wd\z@}$}%
    \dp\tw@-\ht\z@
    \@tempdima\ht\z@ \advance\@tempdima2\ht\tw@ \divide\@tempdima\thr@@
    \setbox\tw@\hbox{%
       \raise\@tempdima\hbox{\scalebox{1}[-1]{\lower\@tempdima\box
\tw@}}}%
    {\ooalign{\box\tw@ \cr \box\z@}}}
\makeatother

\makeindex
\crefformat{enumi}{#2\textup{(#1)}#3}

\makeatletter
\newcommand{\acro}[3]{%
  \newacronym{#1}{#2}{#3}%
  \@namedef{#1}{{\acrshort{#1}}\xspace}%
  \@namedef{#2}{{\acrlong{#1}}\xspace}%
}

\acro{qft}{QFT}{Quantum Field Theory}
\acro{tqft}{TQFT}{Topological Quantum Field Theory}
\acro{cft}{CFT}{Conformal Field Theory}
\acro{rcft}{RCFT}{Rational Conformal Field Theory}
\acro{sft}{SFT}{Symplectic Field Theory}
\acro{cohft}{CohFT}{Cohomological Field Theory}
\acro{trft}{TropFT}{Tropical Field Theory}

\acro{syz}{SYZ}{Strominger--Yau--Zaslow}
\acro{tuy}{TUY}{Tsuchiya--Ueno--Yamada}
\acro{wdvv}{WDVV}{Witten--Dijkgraaf--Verlinde--Verlinde}
\acro{wz}{WZ} {Wess--Zumino--Witten}

\acro{glsm}{GLSM}{Gauged Linear Sigma Model}

%% file: macros.tex
\newcommand   \graph {\gamma}          % a graph,  Kohno-style,  != \iter
\newcommand\ii{\ensuremath{\sqrt{-1}}}
\DeclareMathOperator\SU{SU}                   % special unitary group,    SU(2)
\newcommand\Gm{{\ensuremath{\bG_{\mathrm{m}}}}} % multiplicative group G_m

\newcommand\LL{\bL}        % Lefschetz class

\DeclareMathOperator \Gr     {Gr}               % Grassmannian (variety)
\DeclareMathOperator \Jac    {Jac}              % Jacobian (variety)
\DeclareMathOperator \Pic    {Pic}              % Picard (group scheme)
\DeclareMathOperator \moduli {M}                % moduli (space)
\newcommand \odd  {\ensuremath{\moduli_C(2,\cL)}}   % det = L, fixed of odd degree

\DeclareMathOperator \Fuk {Fuk}                 % -//- majuscule v.
\DeclareMathOperator \FS  {FS}                    % Fukaya-Seidel cat. [Α_∞]
\DeclareMathOperator \MF  {MF}                    % Matrix Factorizations cat.

\DeclareMathOperator\derived{\fD}               % [adjective] derived (cat.)
\newcommand\bounded{\ensuremath{\mathrm{b}}}    % [adjective] bounded (complexes, of)
\newcommand\dbcoh[1]{\derived^\bounded(#1)}

\mathchardef\mhyphen="2D
\newcommand\dash{\nobreakdash-\hspace{0pt}}

\newcommand\pt{\ensuremath{\mathrm{pt}}}

\newcommand \Chow  {\ensuremath{\mathrm{Chow}}}   % Chow motives category
\newcommand \dgCat {\ensuremath{\mathrm{dgCat}}}  % dg-categories (2-category of)
\newcommand \Var   {\ensuremath{\mathrm{Var}}}    % Varieties (categoy of)
\newcommand \dd    {\ensuremath{\mathrm{d}}}      % de Rham differential

\DeclareMathOperator\Bl{Bl}              % blowup
\DeclareMathOperator\cc{c}               % Chern class

\DeclareMathOperator\HH{H}
\DeclareMathOperator\HHHH{HH}
\DeclareMathOperator\Hom{Hom}            % homomorphisms
\DeclareMathOperator\Kzero{K_0}          % Grothendieck group
\DeclareMathOperator\QH{QH}              % quantum (co)homology

\DeclareMathOperator\RRHHom{\fR{\cH}om}
\DeclareMathOperator\Sym{Sym}            % symmetric power
\DeclareMathOperator\ZZKap{Z_{Kap}}      % Kapranov "motivic" ζ-function
\newcommand \bA {{\ensuremath{\mathbb{A}}}}

\newcommand \bC {{\ensuremath{\mathbb{C}}}}

\newcommand \bF {{\ensuremath{\mathbb{F}}}}
\newcommand \bG {{\ensuremath{\mathbb{G}}}}

\newcommand \bL {{\ensuremath{\mathbb{L}}}}

\newcommand \bP {{\ensuremath{\mathbb{P}}}}
\newcommand \bQ {{\ensuremath{\mathbb{Q}}}}

\newcommand \bZ {{\ensuremath{\mathbb{Z}}}}

\newcommand \fD {{\ensuremath{\mathbf{D}}}}

\newcommand \fR {{\ensuremath{\mathbf{R}}}}

\newcommand \cA {\ensuremath{\mathcal{A}}}
\newcommand \cB {\ensuremath{\mathcal{B}}}
\newcommand \cC {\ensuremath{\mathcal{C}}}

\newcommand \cE {\ensuremath{\mathcal{E}}}

\newcommand \cH {\ensuremath{\mathcal{H}}}

\newcommand \cL {\ensuremath{\mathcal{L}}}
\newcommand \cM {\ensuremath{\mathcal{M}}}

\newcommand \cO {\ensuremath{\mathcal{O}}}
\newcommand \cP {\ensuremath{\mathcal{P}}}

\newcommand \cR {\ensuremath{\mathcal{R}}}

\newcommand \cW {\ensuremath{\mathcal{W}}}

%% file: decomp-intro.tex
\section{Introduction}
% variation on the introduction from monolith/intro.tex: this decomposition article is the article in which the different algebro-geometric and symplectic incarnations are played out against each other, so this is where highlighting this makes the most sense
% need to make sure the other introductions don't blatantly plagiarise this
The moduli space~$\odd$ of stable rank 2 vector bundles
with fixed determinant of odd degree on a curve~$C$ of genus~$g\geq 2$ is an important and well-studied object,
its history going back to the 1960s \cite{MR237500,MR0234958,MR0184252}.
It is a smooth projective Fano variety of dimension~$3g-3$ when considered as an algebro-geometric object,
and a monotone symplectic manifold of dimension~$6g-6$ when considered as a symplecto-geometric object.
They also arise as moduli spaces of flat unitary~$\SU(2)$-connections on Riemann surfaces,
or character varieties of conjugacy classes.

Associated to~$\odd$ one can study its invariants,
such as its Betti numbers \cite{MR364254},
or Chow motive \cite{MR1817504}.
The idea behind the description of the Chow motive
is to find a \emph{decomposition}
which expresses it in terms of easier objects,
that still possess a connection to the curve~$C$.

In this article we will discuss 3 types of decompositions,
which have (conjectural) relations to each other:
% this allows the reader to skip ahead to their favourite decomposition
\begin{itemize}
  \item In \cref{conjecture:ur-sod} we propose a \emph{semiorthogonal} decomposition of the derived category~$\dbcoh{\odd}$.
  \item In \cref{theorem:critical-loci} we describe the critical values and critical loci of the graph potentials from \cite{gp-tqft,gp-sympl}:
    this describes \emph{orthogonal} decompositions of the Fukaya category resp.~the category of matrix factorisations.

    When seen as candidate Landau--Ginzburg mirrors to~$\odd$,
    this description is consistent with
    the eigenvalue decomposition of the quantum multiplication~$\mathrm{c}_1(\odd)*_0-$
    on quantum cohomology~$\QH^\bullet(\odd)$ following Mu\~noz as described in \cref{proposition:munoz}.
  \item In \cref{theorem:kzero} we describe an identity in~$\Kzero(\Var/k)$,
    the Grothendieck ring of varieties,
    which relates to various known and expected \emph{motivic} decompositions.
\end{itemize}
The decompositions discussed in this article for~$\moduli_C(2,\mathcal{L})$
can serve as a blueprint for other decompositions of moduli spaces of sheaves,
and we will discuss related work in \cref{subsection:sod,subsection:kzero-and-measures}.

\paragraph{Semiorthogonal decompositions of the derived category}
%Bondal--Kapranov \cite{MR1039961}
%introduced the notion of a semiorthogonal decomposition,
%as a means to decompose a triangulated category.
%Later, Bondal--Orlov studied these semiorthogonal decompositions in algebraic geometry,
%see \cite{alg-geom/9506012,MR1957019} and references therein.
%A more recent overview is given in Kuznetsov's ICM address \cite{MR3728631}.
%
Semiorthogonal decompositions encode many important geometric properties of varieties (see Kuznetsov's ICM address \cite{MR3728631} for an introduction),
and nowadays play an important role in the geometry and construction of moduli spaces of sheaves.
In this light we propose the following conjecture.
\begin{alphaconjecture}
\label{conjecture:ur-sod}
Let~$C$ be a smooth projective curve of genus~$g\geq 2$.
Then there exists a semiorthogonal decomposition
\begin{equation}
\label{equation:ur-sod}
\begin{aligned}
&\dbcoh{\odd}
=
\big\langle
\dbcoh{\pt},
\dbcoh{C},
\dbcoh{\Sym^2 C},
\ldots \\
& \qquad\ldots,
\dbcoh{\Sym^{g-2} C},
\dbcoh{\Sym^{g-1} C},
\dbcoh{\Sym^{g-2} C},
\ldots \\
& \qquad\ldots,
\dbcoh{\Sym^2 C},
\dbcoh{C},
\dbcoh{\pt}
\big\rangle.
\end{aligned}
\end{equation}
\end{alphaconjecture}
This was conjectured independently by Narasimhan
(as communicated in \cite{1806.11101})
and the authors
(see \cite[Conjecture~7]{oberwolfach}). We also give a further refinements of this conjecture
in \cref{subsection:lefschetz-decomposition,subsection:ringel-samokhin-decomposition}:
we suggest that there exists a
\begin{itemize}
  \item a minimal Lefschetz decomposition in the sense of \cite{MR2354207} with respect to the generator of the Picard group~$\mathcal{O}_{\odd}(1)$;
  \item a Ringel--Samokhin-type decomposition (which we will introduce), which involves a symmetry under an anti-involution.
\end{itemize}
We will explain the state-of-the-art for these in \cref{subsection:sod}.

\paragraph{Eigenvalue and critical value decompositions}
Under mirror symmetry~$\odd$ corresponds to a Landau--Ginzburg model~$f\colon Y\to\mathbb{A}^1$
where~$Y$ is a quasiprojective variety and~$f$ a regular function on it.
In the introduction to \cref{section:qh-gp} we recall the motivation to describe on one hand
the eigenvalues of the quantum multiplication~$\mathrm{c}_1(\odd)*_0-$
and on the other hand the critical values of~$f$.

In \cite{gp-tqft} we have introduced \emph{graph potentials},
and in \cite{gp-sympl} we explained how they can be seen as the first step in the construction of the mirror Landau--Ginzburg model.
The following theorem summarises the situation, and gives further evidence for \cref{conjecture:ur-sod}
and the claim that graph potentials are (partial) mirror Landau--Ginzburg mirrors to~$\odd$.
\begin{alphatheorem}
  \label{theorem:critical-loci}
  The critical values of graph potential for the necklace graph in \cref{figure:necklace-graph}
  are given by
  \begin{equation}
    \label{equation:critical-values}
    8(1-g),8(2-g)\ii,8(3-g),\ldots,0,\ldots,8(g-3),8(g-2)\ii,8(g-1).
  \end{equation}
  The dimension of the critical loci with modulus~$8g-8-8k$ is equal to~$k$, where~$k=0,\ldots,g-1$.
  For trivalent graphs with a perfect matching the critical values form a subset of \eqref{equation:critical-values}.
%  The eigenvalues of the quantum multiplication~$\mathrm{c}_1(\odd)*_0-$ on quantum cohomology~$\QH^\bullet(\odd)$
%  and the critical values of the graph potentials from \cite{gp-tqft,gp-sympl}\checkthis{maybe we only prove that these are critical values, not that there are any other?}
%  are both given by
%  \begin{equation}
%    8(1-g),8(2-g)\ii,8(3-g),\ldots,0,\ldots,8(g-3),8(g-2)\ii,8(g-1).
%  \end{equation}
%  \checkthis{can we say something about the size of critical loci?}
\end{alphatheorem}
The important conclusions are that graph potentials already see all the (expected) critical values,
even if they are only restrictions of the full Landau--Ginzburg mirror to certain torus charts,
and that the critical loci are of the expected dimension.

This result mirrors
the eigenvalue decomposition of quantum cohomology~$\QH^\bullet(\odd)$ described in \cref{proposition:munoz},
as obtained by Mu\~noz.
It gives further evidence to the claim that graph potentials can be used as
building blocks for the full Landau--Ginzburg mirror of~$\odd$.

\paragraph{The Grothendieck ring of varieties}
The third type of decomposition we consider is an identity for $\odd$ in the Grothendieck ring of varieties. This ring encodes the cut-and-paste relation, and because many invariants of varieties satisfy such a cut-and-paste relation it can be seen as the universal invariant encoding this. As such it implies various known identities for motivic invariants of $\odd$, and as we explain in \cref{corollary:kzero-categories} it is consistent with \cref{conjecture:ur-sod}. The result is the following.

\begin{alphatheorem}
  \label{theorem:kzero}
  We have the equality
  \begin{equation}
    \label{equation:class-in-K_0}
    [\odd] = \LL^{g-1} [\Sym^{g-1}C] + \sum_{i=0}^{g-2}(\LL^i+\LL^{3g-3-2i})[\Sym^iC] + T
  \end{equation}
  in~$\Kzero(\Var/k)$, for some class~$T$ such that~$(1+\LL) \cdot T = 0$.
\end{alphatheorem}
We expect that~$T=0$, but our method of proof is not strong enough to remove this error term.

%\paragraph{Structure of the article}
%In \cref{section:sod} we elaborate \cref{conjecture:ur-sod} and discuss amplifications of the conjectural semiorthogonal decomposition with additional symmetries. We will also describe some known special cases.
%
%In \cref{section:qh-gp} we will explain how Mu\~noz's results on quantum cohomology are related under mirror symmetry to decompositions of critical values of the candidate Landau--Ginzburg mirror constructed in \cite{gp-tqft,gp-sympl}.
%
%Finally, in \cref{section:kzero} we give the identity from \cref{theorem:kzero} in the Grothendieck ring of varieties, which induces an identity in the Grothendieck ring of categories compatible with \cref{conjecture:ur-sod}.

\paragraph{Acknowledgements}
We want to thank
Emanuele Macr\`i,
M.S. Narasimhan,
and Maxim Smirnov
for interesting discussions.

This collaboration started in Bonn in January--March 2018 during
the second author's visit to the
``Periods in Number Theory, Algebraic Geometry and Physics'' Trimester Program
of the Hausdorff Center for Mathematics (HIM)
and the first and third author's stay in
the Max Planck Institute for Mathematics (MPIM),
and the remaining work was done in
the Tata Institute for Fundamental Research (TIFR)
during the second author's visit in December 2019--March 2020
and the first author's visit in February 2020.
We would like to thank
HIM, MPIM and TIFR
for the very pleasant working conditions.

The first author was partially supported by the FWO (Research Foundation---Flanders).
The third author was partially supported by
the Department of Atomic Energy, India, under project no. 12-R\&D-TFR-5.01-0500
and also by the Science and Engineering Research Board, India (SRG/2019/000513).

%% file: decomp-sod.tex
% !TEX encoding = UTF-8 Unicode
% \section{Semiorthogonal decomposition for moduli of rank two bundles}
\section{Semiorthogonal decompositions}
\label{section:sod}
In this section we will discuss the structure
of the bounded derived category of coherent sheaves~$\dbcoh{\odd}$\index{D@$\dbcoh{X}$}.
%by introducing a conjectural semiorthogonal decomposition,
%given as \cref{conjecture:ur-sod} in the introduction.
We also suggest two semiorthogonal decompositions with additional structure,
namely a Lefschetz decomposition
(as introduced by Kuznetsov \cite{MR2354207}) % Kuznetsov
and we also introduce the notion of
a Ringel--Samokhin decomposition,
related to the structure of quasi-hereditary algebras
and an informal analogy between moduli spaces of bundles
with flag varieties.

\subsection{On the conjecture}
\label{subsection:sod}

In \cref{conjecture:ur-sod} we have formulated a conjecture
giving a semiorthogonal decomposition
into geometrically meaningful pieces,
namely copies of~$\dbcoh{\Sym^i C}$ for~$i=0,\ldots,g-1$.
This was conjectured independently by Narasimhan
(as communicated in \cite{1806.11101})
and the authors (see \cite[Conjecture~7]{oberwolfach}).
The conjecture does not specify the embedding functors.
%As discussed later,
%for the copy of~$\dbcoh{C}$ one natural candidate
%is given by the universal vector bundle~$\cW$ on~$C\times\odd$,
%but for higher symmetric powers there is no candidate yet.

Moreover, it was a folklore conjecture
that each of the pieces~$\dbcoh{\Sym^iC}$ for~$i=1,\ldots,g-1$ is indecomposable,
i.e.~does not admit further semiorthogonal decompositions.
This was proven for~$i\leq\lfloor\frac{g+3}{2}\rfloor-1$
in \cite[Corollary~D]{2007.00994},
with a weaker result being given in \cite[Theorem~1.3]{MR4276298},
and finally for all~$i\leq g-1$ in \cite[Theorem~1.4]{2107.09564v1}.

% In \cref{subsection:lefschetz-decomposition,subsection:ringel-samokhin-decomposition}
% we will further refine this conjecture,
% by putting conditions on the embedding functors,
% giving extra structure on the semiorthogonal decomposition.

% In \cref{subsection:sod-vs-hms} we will discuss evidence for this conjecture
% from the point of view of homological mirror symmetry.
% If one can find an equivalence between~$\dbcoh{\odd}$
% and the Fukaya--Seidel category of an appropriately chosen \LG model,
% invariants associated to the \LG model will reflect
% the conjectured semiorthogonal decomposition.
% Guided by a variation of Dubrovin's and \Gamma-conjectures,
% we will rather look at the orthogonal decomposition
% for the quantum cohomology algebra and for the Fukaya category of~$\odd$ and
% for the (conjecturally equivalent) singularity category of the \LG model.
% These categories have decompositions indexed by
% the eigenvalues of quantum multiplication with~$\mathrm{c}_1$
% resp.~the critical values of the potential, and these we can compute, and compare.

% Outside the realm of homological mirror symmetry
% we prove in \cref{section:K_0-decomposition}
% an identity in the Grothendieck ring of varieties
% and its corollary in Bondal--Larsen--Lunts ring of dg categories,
% which would also follow from the semiorthogonal decomposition~\eqref{equation:ur-sod}.

\paragraph{State of the art}
The first steps
(before the precise conjecture was phrased)
were taken
by Narasimhan \cite{MR3713871,MR3880395},
and independently
by Fonarev--Kuznetsov \cite{MR3764066} (for $C$ generic),
who have shown that
the Fourier--Mukai functor~$\Phi_{\cW}$,
where~$\cW$\index{W@$\cW$} is
the universal vector bundle on~$C\times\odd$
is fully faithful.
Together with the exceptional objects~$\cO_{\odd}$ and~$\Theta$
this gives rise to \emph{3 components} in~\eqref{equation:ur-sod}.

The first and third author
have shown in \cite{MR3954042} that it is possible to twist the embedding~$\Phi_{\cW}$ by~$\Theta$
to obtain the embedding of \emph{the four low-dimensional components}
in~\eqref{equation:ur-sod}
(for $g\geq 12$).
The embedding of a copy of~$\derived^\bounded(\Sym^2C)$
was recently obtained by Narasimhan--Lee in \cite[Theorem~1.2]{2106.04872v1}
(for~$C$ not hyperelliptic and~$g\geq 16$).
%More precisely, Theorem~B of op.~cit gives the semiorthogonal decomposition
%\begin{equation}
%  \dbcoh{\odd}
%  =
%  \langle
%    \cO_{\odd},
%    \Phi_{\cW}(\dbcoh{C}),
%    \Theta,
%    \Phi_{\cW}(\dbcoh{C})\otimes\Theta,
%    \cA
%  \rangle
%\end{equation}
%for~$g\geq 12$, where~$\cA$ is defined as the orthogonal to the first subcategories.
%The bound on~$g$ arises from
%the need for sufficiently strong vanishing results in op.~cit.,
%and conjecturally can be removed.

Most importantly,
in \cite{2108.11951v3} and \cite{2108.13353v2}
a semiorthogonal decomposition that contains \emph{all} expected components
was obtained by Tevelev--Torres and Xu--Yau.
The essential ingredient in the first approach
is Thaddeus' wall-crossing picture, which is also used in \cref{section:kzero} to prove \cref{theorem:kzero}.
The second approach uses Teleman's Borel--Weil--Bott theory on the moduli stack of $\mathrm{SL}_2$\dash bundles.
What remains to be proven is that the complement to the subcategory generated by these components is trivial.

\paragraph{Outlook}
This conjecture forms a part of a greater program,
which aims at finding
natural semiorthogonal decompositions
of the bounded derived categories of coherent sheaves
on the moduli spaces of (stable) objects
in the bounded derived categories of coherent sheaves
on algebraic varieties.

For moduli of higher rank vector bundles there exist various motivic decompositions (see also \cref{subsection:kzero-and-measures})
and these suggest semiorthogonal decompositions of~$\dbcoh{\moduli_C(r,\mathcal{L})}$ with~$\gcd(r,\deg\mathcal{L})=1$.
In this case~$\moduli_C(r,\mathcal{L})$ is a smooth projective Fano variety of dimension~$(r^2-1)(g-1)$ and index~$2$.
The motivic decompositions suggest that derived categories of products of~$\Sym^iC$ will play a role,
and for~$r=3$ this was made precise in \cite[Conjecture 1.9]{2007.06067v1}.

Related to this is the indecomposability of~$\derived^\bounded(\Sym^iC)$ for all~$i\leq g-1$
which is proven in \cite[Theorem~1.4]{2107.09564v1}.
The semiorthogonal decomposition of the derived categories of~$\Sym^iC$ for~$i\geq g$
was obtained by Toda in \cite{MR4244515},
and studied again in \cite[\S1.4]{1811.12525} and \cite[Theorem~D]{1909.04321}.

Another important class of moduli spaces of coherent sheaves
are punctual Hilbert schemes.
The study of (partial) semiorthogonal decompositions was started in \cite{MR3397451}.
It was generalized to higher dimensions in \cite{MR3811590,MR3950704}.
Further generalisations to nested Hilbert schemes were obtained
in \cite[\S3.1.4]{1811.12525} and \cite[Theorem~E]{1909.04321}.
%Moreover, for moduli spaces of coherent sheaves on rational surfaces,
%wall-crossing results give rise to semiorthogonal decompositions as in \cite{MR3652079}.
In a different direction, a semiorthogonal decomposition for
the Hilbert square of a cubic hypersurface
was obtained in \cite{2002.04940}
using the Fano variety of lines.

\subsection{Lefschetz decompositions}
\label{subsection:lefschetz-decomposition}
A more structured form of \cref{conjecture:ur-sod}
is given by that of a Lefschetz decomposition,
a notion introduced by Kuznetsov \cite{MR2354207} % Kuznetsov
in his theory of homological projective duality.

\begin{definition}
Let~$X$ be a smooth projective variety,
and~$\cO_X(1)$ be a line bundle on~$X$.
A \emph{(right) Lefschetz decomposition} of~$\dbcoh{X}$
 with respect to the line bundle~$\cO_X(1)$
is a semiorthogonal decomposition
\begin{equation}
\label{equation:lefschetz-decomposition}
\dbcoh{X}
=
\langle
\cA_0,
\cA_1 \otimes \cO_X(1) ,
\ldots,
\cA_{n-1} \otimes \cO_X(n-1)
\rangle,
\end{equation}
such that
\begin{equation}
\cA_{n-1} \subseteq \cA_{n-2} \subseteq \ldots \subseteq \cA_0 .
\end{equation}
We say that~$n$ is the \emph{length} of the decomposition,
and~$\cA_0$ is the \emph{starting block}.

The decomposition~\eqref{equation:lefschetz-decomposition}
is \emph{rectangular} if~\( \cA_{n-1} = \ldots = \cA_0 \).
\end{definition}

By \cite[Lemma~2.18(i)]{MR2403307}
a Lefschetz decomposition is completely determined by its starting block,
as one can inductively define
\( \cA_i := {}^\perp(\cA_0\otimes\cO_X(-i)) \cap \cA_{i-1} \).
This allows us to define the notion of
a \emph{minimal} Lefschetz collection,
by considering the (partial) inclusion order
on the starting blocks for Lefschetz collections
of~$\dbcoh{X}$ with respect to~$\cO_X(1)$.

In \cite[Definition~1.3(iii)]{MR4017162}
the notion of \emph{residual category}
for a Lefschetz exceptional collection is introduced,
to measure how far it is from being rectangular.
% Here, a Lefschetz exceptional collection
% is a Lefschetz decomposition for which~$\cA_0$ has a full exceptional collection.
One can generalize this to an arbitrary Lefschetz decomposition as in \cite[Definition~2.6]{MR4017162}.
\begin{definition}
  Let~$\dbcoh{X}=\langle\cA_0,\cA_1\otimes\cO_X(1),\ldots,\cA_{n-1}\otimes\cO_X(n-1)\rangle$ be a Lefschetz decomposition of length~$n$ with respect to the line bundle~$\cO_X(1)$. The \emph{residual category}~$\cR$\index{R@$\cR$} of this decomposition is defined as
  \begin{equation}
    \cR:=\langle\cA_{n-1},\cA_{n-1}\otimes\cO_X(1),\ldots,\cA_{n-1}\otimes\cO_X(n-1)\rangle^\perp;
  \end{equation}
\end{definition}

\begin{remark}
  \label{remark:kuznetsov-smirnov-for-odd-cohomology}
  As discussed in \cite{MR4264080} (in the context of Fano varieties with vanishing odd cohomology), properties of the residual category are conjecturally related to the structure of fiber over zero for the quantum cohomology of~$X$. For Fano varieties with odd cohomology, one aspect of a generalization of this conjecture is the prediction that the dimension of the Hochschild homology of the residual category is equal to the length of the eigenspace associated to the eigenvalue~0. By \cref{proposition:munoz} this eigenspace is isomorphic to~$\HH^\bullet(\Sym^{g-1}C)$ (of dimension~$(2g+1)!/g!^2$), which by the Hochschild--Kostant--Rosenberg decomposition is isomorphic to~$\HHHH_\bullet(\Sym^{g-1}C)$, hence the equality holds.
\end{remark}

This brings us to the following more precise form of \cref{conjecture:ur-sod}, incorporating a Lefschetz structure.

\begin{conjecture}
  \label{conjecture:lefschetz-sod}
  Let~$C$ be a smooth projective curve of genus~$g$. Then there exists a minimal Lefschetz decomposition of~$\dbcoh{\odd}$ of length~2 with respect to~$\Theta$, whose first block has a semiorthogonal decomposition
  \begin{equation}
    \cA_0
    =
    \langle
      \dbcoh{\pt},
      \dbcoh{C},
      \dbcoh{\Sym^2C}
      \ldots,
      \dbcoh{\Sym^{g-1}C}
    \rangle,
  \end{equation}
  so in particular the residual category is~$\dbcoh{\Sym^{g-1}C}$.
\end{conjecture}
%In \cite[Conjecture~1.1]{MR4264080} a further Aut-invariance on the semiorthogonal decomposition is imposed. The automorphism group~$\Aut\odd$ is described in \cite{MR1336336}, and it would be interesting to further understand the role of Aut-invariance and the uniqueness of such semiorthogonal decompositions.

For~$g=2$ this conjecture is known by considering the semiorthogonal decomposition \cite[Theorem~2.9]{alg-geom/9506012} and mutating the object~$\cO_{\odd}(-1)$ to the right.
\begin{proposition}[Bondal--Orlov]
  \label{proposition:bondal-orlov-lefschetz-decomposition}
  Let~$C$ be a curve of genus~2. Then
  \begin{equation}
    \dbcoh{\odd}
    =
    \langle
      \cO_{\odd},
      \Phi_{\cW}(\dbcoh{C}),
      \Theta
    \rangle.
  \end{equation}
  is a minimal Lefschetz decomposition, whose residual category is equivalent to~$\dbcoh{C}$.
\end{proposition}

%\begin{remark}
%  For~$g=3$ one would like to apply \cite[Theorem~B]{MR3954042}, because this would immediately imply the conjecture (up to the identification of the residual category). Unfortunately, in op.~cit.~the lower bound~$g_0'$ (depending on the rank~$r$) is equal to~$12$, so the theorem does not apply. It is necessary to improve the cohomology vanishing and codimension estimates in op.~cit.~to prove the conjecture for~$g=3$.
%\end{remark}

\paragraph{Homological projective duality}
After constructing a (minimal) Lefschetz decomposition the next step would be to understand the homological projective dual of~$\odd$ as in \cite{MR2354207}. This would give the structure of the derived categories of linear sections of~$\odd$, and in particular for a smooth hyperplane section. By Theorem~1.1 of op.~cit. we have that \cref{conjecture:lefschetz-sod} gives a semiorthogonal decomposition
\begin{equation}
  \begin{aligned}
    \dbcoh{\odd\cap H}&=\langle\cC_H,\cA_1(1)\rangle \\
    &=\langle\cC_H;\dbcoh{\pt},\dbcoh{C},\ldots,\dbcoh{\Sym^{g-2}C}\rangle
  \end{aligned}
\end{equation}
To indicate the complexity of the homological projective dual, observe that for~$g=2$ the category~$\cC_H$ consists of~7~exceptional objects (because it's the complement to~$\mathcal{O}(1)$ in the derived category of~$\Bl_5\mathbb{P}^2\cong Q\cap Q'\cap H$ in~$\mathbb{P}^5$).

For~$g=3$ one can use the description from \cite{MR0429897} of~$\odd$ for a hyperelliptic curve~$C$ as the zero locus of~$(\Sym^2\mathcal{U}^\vee)^{\oplus 2}$ on~$\Gr(g-1,2g+2)$, together with the Borel--Weil--Bott theorem and the Koszul sequence, to compute that the Hodge diamond of a hyperplane section is given by
\begin{equation}
  \begin{tabular}{cccccccccccc}
      &     &     &     &      & $1$ &      &     &     &     &     \\
      &     &     &     & $0$  &     & $0$  &     &     &     &     \\
      &     &     & $0$ &      & $1$ &      & $0$ &     &     &     \\
      &     & $0$ &     & $3$  &     & $3$  &     & $0$ &     &     \\
      & $0$ &     & $0$ &      & $2$ &      & $0$ &     & $0$ &     \\
  $0$ &     & $1$ &     & $61$ &     & $61$ &     & $1$ &     & $0$ \\
      & $0$ &     & $0$ &      & $2$ &      & $0$ &     & $0$ &     \\
      &     & $0$ &     & $3$  &     & $3$  &     & $0$ &     &     \\
      &     &     & $0$ &      & $1$ &      & $0$ &     &     &     \\
      &     &     &     & $0$  &     & $0$  &     &     &     &     \\
      &     &     &     &      & $1$ &      &     &     &     &     \\
  \end{tabular}
\end{equation}
Using the Hochschild--Kostant--Rosenberg decomposition and additivity of Hochschild homology, we see that
\begin{equation}
  \operatorname{HH}_\bullet(\mathcal{C}_H)=k[-3]\oplus k^{64}[-1]\oplus k^5\oplus k^{64}[1]\oplus k[3].
\end{equation}
There is no obvious modular interpretation for (a semiorthogonal decomposition of)~$\cC_H$, nor does it look like that the homological projective dual has an easy description.

\subsection{Ringel--Samokhin-type decompositions}
\label{subsection:ringel-samokhin-decomposition}
There is another way in which we can impose further conditions on a semiorthogonal decomposition for~$\dbcoh{\odd}$. This condition is inspired by the theory of hereditary algebras and Ringel duality, and by similar semiorthogonal decompositions for~$\dbcoh{G/B}$ obtained by Samokhin, and encodes special symmetries not found in most semiorthogonal decompositions.

Let~$X$ be a smooth projective variety, and consider a semiorthogonal decomposition
\begin{equation}
  \label{equation:ringel-samokhin-A}
  \dbcoh{X}
  =
  \langle\cA_0,\ldots,\cA_n\rangle.
\end{equation}
of length~$n+1\geq 2$.
Let~$\sigma$ be an anti-equivalence of~$\dbcoh{X}$ which is moreover an involution. Let~$\cB_{k}$ be the image of~$\cA_{n-k}$ under this involution, for~$k=0,\ldots,n$. Then we have a semiorthogonal decomposition
\begin{equation}
  \label{equation:ringel-samokhin-B}
  \dbcoh{X}
  =
  \langle\cB_0,\ldots,\cB_n\rangle.
\end{equation}
On the other hand consider the left dual decomposition in the sense of \cite[Definition~3.6]{1410.3742} denoted by
\begin{equation}
  \label{equation:ringel-samokhin-C}
  \dbcoh{X}
  =
  \langle\cC_0,\ldots,\cC_n\rangle.
\end{equation}
%This decomposition is obtained as follows: First define $\cC_n$ to be the image of mutation of $\cA_{0}$ from left to right. Then $\cC_{n-1}$ is the mutation of $\cA_{1}$ from left to right and continue this process with $\cA_2$ and so on. In particular $\cC_{n-k}$ is the obtained by $n-k$-mutations of $\cA_{k}$.
This brings us to the following definition.
\begin{definition}
  A decomposition~\eqref{equation:ringel-samokhin-A} is of \emph{Ringel--Samokhin type} if there exists an anti-equivalence~$\sigma$ of~$\dbcoh{X}$ such for all~$k=0,\ldots,n$ the subcategory~$\cC_k$ is the image of~$\cB_k$ under~$\sigma$.
\end{definition}

\begin{remark}
  In \cite{1410.3742} Samokhin studied such exceptional collections for the derived category~$\dbcoh{G/B}$, where~$G/B$ is the full flag variety of a group of rank~2, and the anti-equivalence is given by~$\RRHHom(-,\omega_{G/B}^{1/2})$.
\end{remark}

We can refine \cref{conjecture:ur-sod} into a semiorthogonal decomposition of Ringel--Samokhin-type as follows.
\begin{conjecture}
  \label{conjecture:ringel-samokhin-sod}
  Let~$C$ be a smooth projective curve of genus~$g\geq 2$. Then there exists a Ringel--Samokhin-type decomposition
  \begin{equation}
  \begin{aligned}
    \dbcoh{\odd}
    =
    \langle
      \dbcoh{\pt},
      \dbcoh{C},
      \ldots,
      \dbcoh{\Sym^{g-2}C},
      \dbcoh{\Sym^{g-1}C},
      \dbcoh{\Sym^{g-2}C},
      \ldots, \\
      \ldots,\dbcoh{C},
      \dbcoh{\pt}
    \rangle
    \end{aligned}
  \end{equation}
  with the anti-equivalence given by~$\RRHHom(-,\cO_{\odd}(1))$.
\end{conjecture}
This conjecture highlights how semiorthogonal decompositions of~$\odd$ are expected to have strong symmetry properties, much stronger than found for most varieties exhibiting semiorthogonal decompositions.

For~$g=2$ we can illustrate this conjecture using the following general result for Ringel--Samokhin-type decompositions of length~3.

\begin{proposition}
  \label{prop:BOS}
  Let~$X$ be a smooth projective variety. Assume that we have a semiorthogonal decomposition
  \begin{equation}
    \label{equation:ringel-samokhin-length-3}
    \dbcoh{X}
    =
    \langle
      \cL_1,
      \cA,
      \cL_2
    \rangle
  \end{equation}
  where~$\cL_1,\cL_2$ are exceptional line bundles. Then this decomposition is of Ringel--Samokhin-type if and only if~$\cL_1\otimes \cL_2^{\vee}$ is a theta characteristic of~$X$.
\end{proposition}

\begin{proof}
  A decomposition of the form~\eqref{equation:ringel-samokhin-length-3} is determined by the line bundles~$\cL_1$ and~$\cL_2$, as the subcategory~$\cA$ consists of the objects~$E\in\dbcoh{X}$ such that
  \begin{equation}
  \Hom^\bullet(\RRHHom(E,\cO_X),\cL_1)=\Hom^\bullet(\cL_2,E)=0.
  \end{equation}

  The left dual decomposition of~\eqref{equation:ringel-samokhin-length-3} is given by
  \begin{equation}
    \dbcoh{X}
    =
    \langle
      \cL_2,
      \mathbf{L}_{\cL_2}\cA,
      \cL_1\otimes\omega_X^\vee
      \rangle.
  \end{equation}
  On the other hand, we consider the anti-equivalence given by dualizing and tensoring with a line bundle~$\cM$, to obtain the semiorthogonal decomposition
  \begin{equation}
    \dbcoh{X}
    =
    \langle
      \cL_2^\vee\otimes\cL,
      \cA^\vee\otimes\cL,
      \cL_1^\vee\otimes\cL
    \rangle.
  \end{equation}
  These two semiorthogonal decompositions agree if and only if~$\cL\cong\cL_2^{\otimes 2}$, so that we have~$(\cL_1^\vee\otimes\cL_2)^{\otimes2}\cong\omega_X^\vee$ and hence~$\cL_1\otimes\cL_2^\vee$ is a theta characteristic of~$X$.
\end{proof}
By \cref{proposition:bondal-orlov-lefschetz-decomposition},
using that~$\odd$ is of index~2, we obtain the following corollary.
\begin{corollary}
  \Cref{conjecture:ringel-samokhin-sod} holds for~$g=2$.
\end{corollary}

%% file: decomp-qh-gp.tex
\section{Eigenvalue and critical value decompositions}
\label{section:qh-gp}
To understand the results in this section
we will briefly recall
homological mirror symmetry
for Fano varieties.
If~$X$ is a (smooth projective) Fano variety,
its mirror is expected to be
a Landau--Ginzburg model~$f\colon Y\to\mathbb{A}^1$
(where~$Y$ is a quasiprojective variety),
such that we have equivalences
of triangulated categories
\index{F@$\Fuk(X)$}\index{M@$\MF(Y,f)$}\index{F@$\FS(Y,f)$}
\begin{equation}
  \begin{aligned}
    \Fuk(X)&\cong\MF(Y,f), \\
    \dbcoh{X}&\cong\FS(Y,f).
  \end{aligned}
\end{equation}
On the first line we have that the Fukaya category of~$X$ has an \emph{orthogonal decomposition}, indexed by the eigenvalues of~$\mathrm{c}_1(X)*_0-$, see \cite{MR2929086}. For the \emph{matrix factorization category} we have an orthogonal decomposition, indexed by the critical values of~$f$ \cite[Proposition~1.14]{MR2101296}. Hence under the homological mirror symmetry conjecture these sets (of eigenvalues, and critical values) are the same.

For the second line there is no such natural (semiorthogonal) decomposition. But the philosophy behind Dubrovin's conjecture (which a priori is only formulated in the case of semisimple quantum cohomology) predicts how an orthogonal decomposition of the triangulated category~$\Fuk(X)$ gives rise to a semiorthogonal decomposition of~$\dbcoh{X}$. Generalizations of this philosophy to not necessarily semisimple quantum cohomology are discussed in \cite{MR4105948}.

Hence to understand semiorthogonal decompositions for~$\dbcoh{\odd}$ using mirror symmetry, one can study either of the following
\begin{enumerate}
  \item the eigenvalues of the quantum multiplication~$\mathrm{c}_1(\odd)*_0-$ on the quantum cohomology~$\QH^\bullet(\odd)$;\index{Q@$\QH^\bullet(X)$}
  \item the critical values of the potential~$f\colon Y\to\mathbb{A}^1$.
\end{enumerate}

\begin{remark}
  \label{remark:no-mirror}
  There is no worked out candidate construction for the Landau--Ginzburg mirror~$(Y,f)$ of~$\odd$ yet,
  let alone a proof of homological mirror symmetry.
  But the graph potentials we have constructed and studied
  in \cite{gp-tqft,gp-sympl}
  can be seen as open cluster charts of~$Y$,
  and gluing these tori together
  along the birational transformations between them
  (which we can do because of the compatibilities from \cite[Theorems~2.12 and 2.13]{gp-tqft})
  is a first step one can take
  in the construction of the Landau--Ginzburg model mirror to~$\odd$.
  We will recall this construction in \cref{subsection:graph-potentials}.
  %, but the morphism is not yet proper.
\end{remark}

% But it turns out that one can avoid this gluing step if one is only interested in the critical values of the potential, as it suffices to consider \emph{one cluster chart at a time}: every graph potential sees all of the critical values. In this section we will discuss the literature on the eigenvalues of quantum multiplication, and compare them to the critical values of graph potentials obtained in \cref{subsection:critical-values}.

A better understanding of the critical loci
would be the next step in understanding if and how
these graph potentials can be glued together to obtain (part of) the Landau--Ginzburg model.
In \cref{section:critical} we discuss the dimension of the different critical loci
as a first step in this program.

\subsection{Mu\~noz's eigenvalue decomposition for quantum cohomology}
The eigenvalues of quantum multiplication with~$\mathrm{c}_1$ have been described, albeit via an indirect route. The following summarizes this description.
\begin{proposition}[Mu\~noz]
  \label{proposition:munoz}
  The eigenvalue decomposition for the quantum multiplication by~$\cc_1(\odd)$ on~$\QH^\bullet(\odd)$ is
  \begin{equation}
    \QH^\bullet(\odd)
    =
    \bigoplus_{m=1-g}^{g-1}H_m
  \end{equation}
  where
  \begin{enumerate}
    \item the eigenvalues are~$8(1-g),8(2-g)\ii,8(3-g),\ldots,8(g-3),8(g-2)\ii,8(g-1)$;
    \item $H_m$ is isomorphic (as a vector space) to~$\HH^\bullet(\Sym^{g-1-|m|}C)$.
  \end{enumerate}
\end{proposition}

\begin{proof}
  This is a combination of various results of Mu\~noz. First we have \cite[Proposition~20]{MR1670396}, describing the eigenvalues of multiplication by the generator of the Picard group on the instanton Floer homology of the 3-manifold given by the product of the curve (seen as a real manifold) and~$S^1$. The (conjectural) identification as rings from \cite[Theorem~1]{MR1670396} with the quantum cohomology of~$\odd$ is in turn given by \cite[Corollary~21]{MR1695800} after an explicit description in terms of generators and relations for both rings. Hence we can interpret any result for instanton Floer homology as a result for quantum cohomology. In \cite[Conjecture~24]{MR1670396} the conjectural decomposition was given, and in \cite[Corollary~3.7]{MR1828464} it was proved.
\end{proof}

\paragraph{Property $\cO$}
An interesting symmetry property for the quantum cohomology of a Fano variety~$X$ of index~$r$, is obtained by considering the exceptional collection~$\cO_X,\ldots,\cO_X(r-1)$ in~$\dbcoh{X}$. The existence of this exceptional collection can be encoded in quantum cohomology using \cite[Definition~3.1.1]{MR3536989} as follows.
\begin{definition}
  Let~$X$ be a smooth projective Fano variety, of index~$r\geq 1$. Define\index{T@$T$}
  \begin{equation}
    T:=\max\{|u|\mid \text{$u\in\bC$ is an eigenvalue of~$\mathrm{c}_1(X)*_0-$}\}\in\overline{\bQ}_{\geq 0}.
  \end{equation}
  Then we say that~$X$ has \emph{property~$\cO$} if
  \begin{enumerate}
    \item $T$ is an eigenvalue of~$\mathrm{c}_1(X)*_0-$, of multiplicity~1;
    \item if~$u$ is another eigenvalue of~$\mathrm{c}_1(X)*_0-$ such that~$|u|=T$, then there exists a primitive~$r$th root of unity~$\zeta$ such that~$u=\zeta T$.
  \end{enumerate}
\end{definition}
In \cite[Conjecture~3.1.2]{MR3536989} it was conjectured that this property holds for all Fano varieties. In \cite{MR3581315} it was checked for all homogeneous varieties~$G/P$, and in \cite[Corollary~7.7]{MR4105948} the case of complete intersections of index~$r\geq 2$ in~$\bP^n$ was checked.

Hence by \cref{proposition:munoz} we immediately obtain the following result.
\begin{corollary}
  \label{corollary:property-O}
  Property~$\cO$ holds for~$\odd$, where~$T=8(g-1)$.
\end{corollary}

%The discussion so far focussed on the two exceptional objects~$\cO_{\odd}$ and~$\cO_{\odd}(1)$. In \cite{1705.05989} a generalisation of Dubrovin's conjecture
%
%\begin{remark}
%  The systematic~$\bZ/2\bZ$\dash symmetry in the eigenvalues is precisely what underlies the conjectural Lefschetz decomposition from \cref{conjecture:lefschetz-sod}.
%\end{remark}

\subsection{Graph potentials}
\label{subsection:graph-potentials}
Now we consider the other side of the mirror,
and introduce graph potentials as candidate building blocks.
Whilst no construction of the homological mirror exists yet as discussed in \cref{remark:no-mirror}
%in \cite{gp-tqft} we have constructed a family of Landau--Ginzburg models~$W\colon(\mathbb{C}^\times)^{3g-3}\to\mathbb{A}^1$ for every~$g\geq 2$,
%named graph potentials,
%and in \cite{gp-sympl} we have argued that these at least form an enumerative mirror for~$\odd$.
%This family of Landau--Ginzburg models is indexed by colored trivalent graphs,
we will study the geometry of the critical values and loci in \cref{subsection:critical-values}.
Depending on the viewpoint,
this gives further evidence that graph potentials are indeed building blocks for the mirror to~$\odd$,
or it gives further evidence for \cref{conjecture:ur-sod}.

\paragraph{Graph potentials}
We briefly recall the construction from \cite[\S2.1]{gp-tqft}.
Let~$\graph=(V,E)$ be a connected undirected trivalent graph, of genus~$g$.
Thus~$\#V=2g-2$ and~$\#E=3g-3$.
As explained in loc.~cit.
it is natural to define everything in terms of graph (co)homology.
The relevant space is
\begin{equation}
  \widetilde{N}_\graph:=\mathrm{C}^1(\gamma,\mathbb{Z})
\end{equation}
in which a vertex~$v\in V$ which is adjacent to the edges~$e_{v_i},e_{v_j},e_{v_k}$
defines cochains~$x_i,x_j,x_k$ by~$x_i(e_{v_a})=\delta_{i,v_a}$.
% explain graph homology, and variables

We will also decorate trivalent graphs,
by assigning a color to every vertex,
i.e.~we consider~$c\colon V\to\mathbb{F}_2$
where~$c(v)=0$ (resp.~$c(v)=1$) means the vertex~$v$ is uncolored (resp.~colored).
An uncolored vertex is drawn as a circle~\begin{tikzpicture} \node[vertex] (A) at (0,0) {};\end{tikzpicture}
whilst a colored vertex corresponds to a disc~\begin{tikzpicture} \node[vertex,fill] (A) at (0,0) {};\end{tikzpicture}.

\begin{construction}
  \label{construction:graph-potential}
  The \emph{vertex potential}~$\widetilde{W}_{v,c(v)}$ for a vertex~$v\in V$ in a colored graph~$(\gamma,c)$ is the sum of the four monomials
  \begin{equation}
    x_i^{(-1)^{s_i}}x_j^{(-1)^{s_j}}x_k^{(-1)^{s_k}}
  \end{equation}
  where~$(s_i,s_j,s_k)\in\mathbb{F}_2^{\oplus3}$ ranges over all sign choices such that~$s_i+s_j+s_k=c(v)$. Here~$x_i$ is the coordinate variables in~$\mathbb{Z}[\widetilde{N}_\graph]$ corresponding to the~$i$th edge in an enumeration~$e_1,\ldots,e_{3g-3}$ of the edges.

  Hence in the variables~$x,y,z$ there are precisely two cases:
  \begin{equation}
    \begin{aligned}
      \widetilde{W}_{v,0}&=xyz+\frac{x}{yz}+\frac{y}{xz}+\frac{z}{xy} \\
      \widetilde{W}_{v,1}&=\frac{1}{xyz}+\frac{xy}{z}+\frac{xz}{y}+\frac{yz}{x}.
    \end{aligned}
  \end{equation}
  The global structure of the colored graph~$(\gamma,c)$ then defines the \emph{graph potential} as the sum of vertex potentials
  \begin{equation}
    \widetilde{W}_{\graph,c}:=\sum_{v\in V}\widetilde{W}_{v,c(v)}.
  \end{equation}
\end{construction}

\begin{figure}[ht]
  \centering
  \begin{tikzpicture}[scale=1.75]
    \node[vertex] (A) at (0,0) {};
    \node[vertex, fill] (B) at (1,0) {};

    \draw (A) +(0,-0.1) node [below] {$1$};
    \draw (B) +(0,-0.1) node [below] {$2$};

    \draw (A) edge [bend left]  node [above]      {$x$} (B);
    \draw (A) edge              node [fill=white] {$y$} (B);
    \draw (A) edge [bend right] node [below]      {$z$} (B);
  \end{tikzpicture}
  \caption{Colored Theta graph in genus $g=2$}
  \label{figure:theta-graph}
\end{figure}
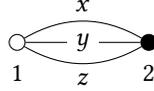

\begin{example}
  In genus~$g=2$ there are precisely two trivalent graphs,
  the Theta graph and the dumbbell graph \cite[Figure~1]{gp-tqft}.
  In \cref{figure:theta-graph} we have given the Theta graph
  with the second vertex colored,
  and following \cref{construction:graph-potential} we get that the graph potential is given by
  \begin{equation}
    \widetilde{W}_{\graph,c}=xyz+\frac{x}{yz}+\frac{y}{xz}+\frac{z}{xy}+\frac{1}{xyz}+\frac{xy}{z}+\frac{xz}{y}+\frac{yz}{x}.
  \end{equation}
\end{example}
More examples are given in \cite[\S2.1]{gp-tqft}.

\paragraph{Elementary transformations}
Trivalent graphs correspond to pair-of-pants decompositions,
and for a given surface there exist many such decompositions up to isotopy.
By Hatcher--Thurston \cite[Appendix]{MR0579573} they are related via certain operations,
and for our purposes we only need operation~(I) of op.~cit.
The operation this induces on trivalent graphs is called an \emph{elementary transformation}
and the local picture is given in \cref{figure:elementary-transformation}

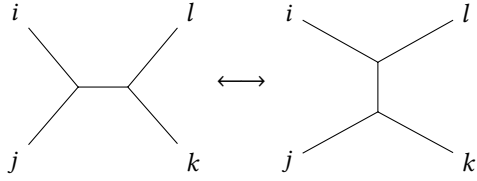
\begin{figure}[t!]
	\centering

	\begin{tikzpicture}[xscale=0.66,baseline=(current bounding box.center)]
	\node (i) at (0,2) [left]  {$i$};
	\node (j) at (0,0) [left]  {$j$};
	\node (k) at (3,0) [right] {$k$};
	\node (l) at (3,2) [right] {$l$};

	\draw (i) -- (1,1) -- (j) (1,1) -- (2,1) -- (k) (2,1) -- (l);
	\end{tikzpicture}
	$\longleftrightarrow$
	\begin{tikzpicture}[yscale=0.66,baseline=(current bounding box.center)]
	\node (i) at (0,3) [left]  {$i$};
	\node (j) at (0,0) [left]  {$j$};
	\node (k) at (2,0) [right] {$k$};
	\node (l) at (2,3) [right] {$l$};

	\draw (i) -- (1,2) -- (l) (1,2) -- (1,1) -- (j) (1,1) -- (k);
	\end{tikzpicture}
	\caption{Local picture of an elementary transformation of a trivalent graph}
	\label{figure:elementary-transformation}
\end{figure}

In \cite[\S2.2]{gp-tqft} we have discussed how elementary transformations
transform the associated graph potentials.
The main results
(given as Corollary~2.9 and Theorems~2.12 and~2.13\ in op.~cit.)
state that
\begin{itemize}
  \item up to biregular automorphism of the torus
    the graph potential only depends on
    the \emph{parity} of the coloring;
  \item up to rational change of coordinates
    the graph potential only depends on
    the \emph{genus} of the trivalent graph.
\end{itemize}

\subsection{Critical value decomposition for graph potentials}
\label{subsection:critical-values}
Our final goal is to prove \cref{theorem:critical-loci}.
In this section we will describe the critical values
and discuss the conifold point of the graph potential.
The computation of the critical loci is deferred to \cref{section:critical}.

\paragraph{Critical values}
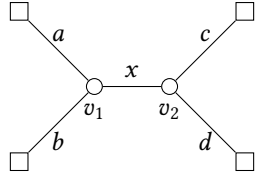
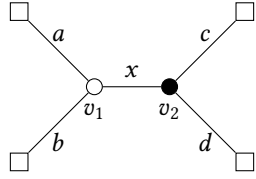
\begin{figure}[t!]
  \centering
  \begin{subfigure}[b]{\textwidth}
    \centering
    \begin{tikzpicture}[scale=1]
      \node[vertex] (first)  at (0,0) {};
      \node[vertex] (second) at (1,0) {};
      \node[draw] (A) at (-1,1)  {};
      \node[draw] (B) at (-1,-1) {};
      \node[draw] (C) at (2,1)   {};
      \node[draw] (D) at (2,-1)  {};

      \draw (first)  +(0,-0.1) node [below] {$v_1$};
      \draw (second) +(0,-0.1) node [below] {$v_2$};

      \draw (A)     edge node [above] {$a$} (first);
      \draw (B)     edge node [below] {$b$} (first);
      \draw (first) edge node [above] {$x$} (second);
      \draw (C)     edge node [above] {$c$} (second);
      \draw (D)     edge node [below] {$d$} (second);
    \end{tikzpicture}
    \caption{Local picture of trivalent graph with $c$ trivial}
    \label{subfigure:local-picture-no-colors}
  \end{subfigure}

  {\ }

  \begin{subfigure}[b]{\textwidth}
    \centering
    \begin{tikzpicture}[scale=1]
      \node[vertex]       (first)  at (0,0) {};
      \node[vertex, fill] (second) at (1,0) {};
      \node[draw] (A) at (-1,1)  {};
      \node[draw] (B) at (-1,-1) {};
      \node[draw] (C) at (2,1)   {};
      \node[draw] (D) at (2,-1)  {};

      \draw (first)  +(0,-0.1) node [below] {$v_1$};
      \draw (second) +(0,-0.1) node [below] {$v_2$};

      \draw (A)     edge node [above] {$a$} (first);
      \draw (B)     edge node [below] {$b$} (first);
      \draw (first) edge node [above] {$x$} (second);
      \draw (C)     edge node [above] {$c$} (second);
      \draw (D)     edge node [below] {$d$} (second);
    \end{tikzpicture}
    \caption{Local picture of trivalent graph with $c$ non-trivial}
    \label{subfigure:local-picture-colors}
  \end{subfigure}
  \caption{Local pictures of colored trivalent graphs}
  \label{figure:local-picture}
\end{figure}

Let~$\graph$ be a trivalent graph, and~$c$ be a coloring.
By \cite[Corollary~2.9]{gp-tqft} we can assume that~$c$ has at most one colored vertex.
To determine the critical values of the graph potential~$\widetilde{W}_{\graph,c}$ we need to determine solutions to the equations
\begin{equation}
  x_i\frac{\partial}{\partial x_i}\widetilde{W}_{\graph,c}=0,\forall i=1,\ldots,3g-3.
\end{equation}

As for the description of the behavior of the graph potential under elementary transformations in the previous section,
we will study the local picture for (colored) trivalent graphs.
These are given in \cref{figure:local-picture},
where \cref{subfigure:local-picture-no-colors}
(resp.~\cref{subfigure:local-picture-colors})
corresponds to the case where the coloring is trivial
(resp.~the vertex~$v_2$ is the unique colored vertex).

We can write the potential in the uncolored (resp.~colored) case as
\begin{equation}
  \label{equation:edge-potential-uncolored}
  \widetilde{W}_{\graph,0}=x\left( ab+\frac{1}{ab}+cd+\frac{1}{cd} \right)+x^{-1}\left( \frac{a}{b}+\frac{b}{a}+\frac{c}{d}+\frac{d}{c} \right)+\widetilde{W}_{\graph,0}^{\text{frozen}}
\end{equation}
resp.
\begin{equation}
  \label{equation:edge-potential-colored}
  \widetilde{W}_{\graph,c}=x\left( ab+\frac{1}{ab}+cd+\frac{1}{cd} \right)+x^{-1}\left( \frac{a}{b}+\frac{b}{a}+cd+\frac{1}{cd} \right)+\widetilde{W}_{\graph,c}^{\text{frozen}}
\end{equation}
where we have split off the frozen part which does not involve the variable~$x$.
We call these expressions (without the frozen part) \emph{edge potentials}.

Hence we are interested in solutions to the equations~$x\frac{\partial}{\partial x}\widetilde{W}_{\graph,0}=0$,
resp.~$x\frac{\partial}{\partial x}\widetilde{W}_{\graph,c}=0$, which can be rewritten by clearing the denominators as
\begin{equation}
  \label{equation:critical-point-uncolored}
  x(1+abcd)(ab+cd)=x^{-1}(ac+bd)(ad+bc)
\end{equation}
resp.
\begin{equation}
  \label{equation:critical-point-colored}
  x(a+bcd)(b+acd)=x^{-1}(c+abd)(d+abc).
\end{equation}

\paragraph{Perfect matchings}
The graph potential is defined as a sum over the vertex potentials associated to the vertices.
But for certain graphs we can alternatively write the graph potential as a sum of edge potentials.
For this we need a perfect matching,
i.e.~a subset~$P\subseteq E$ of edges
such that every vertex is contained in precisely one~$e\in P$.
If the genus of~$\graph$ is~$g$ then~$\#P=g-1$.

Given a perfect matching~$P$ for~$\graph$
we can rewrite the graph potential by summing over the edges in the matching,
and using the expressions~\eqref{equation:edge-potential-uncolored},
\eqref{equation:edge-potential-colored} for the edge potentials.
Let~$e$ denote the edge between~$v_1$ and~$v_2$,
and label the variables in \cref{figure:local-picture} as~$a_e,b_e,c_e,d_e,x_e$.
Then in the uncolored case we have
\begin{equation}
  \label{equation:sum-of-edge-potentials-uncolored}
  \widetilde{W}_{\graph,0}=\sum_{e\in P}x_e\left( a_eb_e+\frac{1}{a_eb_e}+c_ed_e+\frac{1}{c_ed_e} \right)+x_e^{-1}\left( \frac{a_e}{b_e}+\frac{b_e}{a_e}+\frac{c_e}{d_e}+\frac{d_e}{c_e} \right)
\end{equation}
whilst in the colored case we let~$e_c\in E$ denote the colored vertex and omit the subscript~$e_c$ from the variables, so that
\begin{equation}
  \label{equation:sum-of-edge-potentials-colored}
  \begin{aligned}
    \widetilde{W}_{\graph,c}&=x\left( ab+\frac{1}{ab}+cd+\frac{1}{cd} \right)+x^{-1}\left( \frac{a}{b}+\frac{b}{a}+cd+\frac{1}{cd} \right) \\
    &\qquad\sum_{e\in P\setminus\{e_c\}}x_e\left( a_eb_e+\frac{1}{a_eb_e}+c_ed_e+\frac{1}{c_ed_e} \right)+x_e^{-1}\left( \frac{a_e}{b_e}+\frac{b_e}{a_e}+\frac{c_e}{d_e}+\frac{d_e}{c_e} \right).
  \end{aligned}
\end{equation}

By Petersen's theorem, a trivalent graph which is bridgeless (i.e.~we cannot remove an edge to make it disconnected) has at least one perfect matching. Such a graph exists for every genus~$g$. Using \cite[Theorems~2.12 and~2.13]{gp-tqft} we have that the graph potentials for different graphs of the same genus are related via rational changes of coordinates, hence its critical values do not depend on the choice of graph.

\begin{proposition}
  \label{proposition:critical-values}
  Let~$\graph$ be a trivalent graph. Let~$c$ be a coloring with at most one colored vertex. Then the following are critical values of the graph potential~$\widetilde{W}_{\graph,c}$:
  \begin{itemize}
    \item purely real critical values:
      \begin{equation}
        8g-8-16k \text{ for $k=0,\ldots,g-1$};
      \end{equation}
    \item purely imaginary critical values:
      \begin{equation}
        (8g-16-16k)\ii \text{ for~$k=0,\ldots,g-2$}.
      \end{equation}
  \end{itemize}
  The critical value~$0$ is listed twice, depending on the parity of~$g$ as a purely real or a purely imaginary critical value.
\end{proposition}

\begin{proof}
  If we assign~$\pm1$ to the variables associated to the edges in the perfect matching, then it is immediate that assigning~$1$ to all variables associated to edges outside the perfect matching gives a solution to the equations~\eqref{equation:critical-point-uncolored},~\eqref{equation:critical-point-colored}. If we choose~$k$ edges from~$P\subseteq E$ to have value~$-1$, evaluating~\eqref{equation:sum-of-edge-potentials-uncolored},~\eqref{equation:sum-of-edge-potentials-colored} gives the value~$8g-8-16k$.

  If we assign~$\pm\ii$ to the variables associated to the edges in the perfect matching, then it is immediate that assigning~$\ii$ to all variables associated to edges outside the perfect matching gives a solution to the equations~\eqref{equation:critical-point-uncolored},~\eqref{equation:critical-point-colored}. If we choose~$k$ edges from~$P\subseteq E$ to have value~$-\ii$, evaluating~\eqref{equation:sum-of-edge-potentials-uncolored},~\eqref{equation:sum-of-edge-potentials-colored} gives the value~$(8g-16-16k)\ii$.
\end{proof}

%One can also show that the set of critical points used in the analysis is preserved under elementary transformations.
%This follows from the precise formulae in the rational change of coordinates, but we leave the details to the interested reader.

\paragraph{Conifold points}
In \cite[Remark~3.1.6]{MR3536989} a conjecture regarding the value of~$T$ from \cref{corollary:property-O} was suggested,
relating it to critical values of the mirror Landau--Ginzburg model.

If~$W$ is a Laurent polynomial with positive coefficients
such that the origin is contained in the Newton polytope (which is the convex hull of the exponents of monomials with nonzero coefficients),
then it is shown in \cite{1404.7388v1} that there is a \emph{unique} critical point with strictly positive (real) coordinates.
This critical point is called the \emph{conifold point}~$x_{\mathrm{con}}$\index{x@$x_{\mathrm{con}}$}.
We will determine this for the graph potentials.

Define~$T_{\mathrm{con}}:=\widetilde{W}_{\graph,c}(x_{\mathrm{con}})$\index{T@$T_{\mathrm{con}}$} as the value of the potential~$\widetilde{W}_{\graph,c}$ at the conifold point. The conjecture says that
\begin{equation}
  T=T_{\mathrm{con}},
\end{equation}
From the description of the conifold point in \cref{proposition:conifold-point} we can easily check this equality. Indeed, the graph potential~$\widetilde{W}_{\graph,c}$ is the sum over the vertex potentials~$\widetilde{W}_v$ for~$v\in V$, and the evaluation at the conifold point for each of these is equal to~4. There are~$2(g-1)$ vertices, so~$T_{\mathrm{con}}=8(g-1)$. This agrees with the value~$T$ from \cref{corollary:property-O}.

\begin{proposition}
  \label{proposition:conifold-point}
  Let~$\graph$ be a trivalent graph, and let~$c$ be a coloring.
  The conifold point of the graph potential~$\widetilde{W}_{\graph,c}$
  is given by~$(1,\ldots,1) \in (\bC^\times)^{3g-3}$,
  and the value of the graph potential at the conifold point is~$8g-8$.
\end{proposition}

\begin{proof}
  This follows from direct computation:
  for every vertex $v$ we have $\dd W_{v,c}(1,1,1) = 0$ and $W_{v,c} (1,1,1) = 4$,
  hence $\dd W_{\graph,c} = \sum_{v\in V} \dd W_v = 0$ and thus
  \begin{equation}
    W_{\graph,c} (1,\ldots,1) = \sum_{v\in V} W_{v,c} (1,1,1) = 4\#V.
  \end{equation}
\end{proof}

As discussed in the introduction of this section,
the eigenvalues of quantum multiplication should correspond to the critical values of a suitable Landau--Ginzburg model.
There is a priori no reason why it would suffice to consider a single Laurent polynomial,
which is the restriction of the potential to a Zariski-open torus inside~$Y$.

But it turns out that graph potentials already see all the critical values one expects.
This is immediate by comparing \cref{proposition:munoz} to \cref{proposition:critical-values}.
It should be remarked that we have not proven that there are no other critical values,
but we conjecture that these are all.

One necessary feature of Landau--Ginzburg mirrors in homological mirror symmetry for Fano varieties
is that the critical locus of the potential is compact and that there are no critical points at infinity.
Because the critical locus contains higher-dimensional components
we see the need for multiple cluster charts,
as a positive-dimensional proper variety needs multiple affine charts
(so in this case at least~$g$, and hence this grows to $\infty$ as the genus grows).

\begin{remark}
  By \cref{proposition:munoz} we see that the~$\HH^{\bullet}(\Sym^{g-1-k}C)$
  appears as an eigenspace of the quantum multiplication with eigenvalue of absolute value~$8k$.
  On the other hand,
  by \cref{section:critical} we see that the critical loci all have expected dimension
  (for the necklace graph).
  This is in agreement with the appearance of~$\dbcoh{\Sym^{g-1-k}C}$
  in the conjectural semiorthogonal decomposition of~$\dbcoh{\odd}$.
  Further agreement between the two decompositions is provided by the discussion in \cref{section:critical}.
\end{remark}

%% file: decomp-kzero.tex
% !TEX encoding = UTF-8 Unicode
\section{Decomposition in the Grothendieck ring of varieties}
\label{section:kzero}
In this section we compute the class of the moduli space~$\odd$ in the Grothendieck ring of varieties~$\Kzero(\Var/k)$\index{K@$\Kzero(\Var/k)$}, and the main result in this section is a proof of \cref{theorem:kzero}. This identity gives further evidence for \cref{conjecture:ur-sod}, by considering the appropriate motivic measure on~$\Kzero(\Var/k)$ as in \cref{corollary:kzero-categories}.

\subsection{The Grothendieck ring of varieties and motivic measures}
\label{subsection:kzero-and-measures}
Recall that the Grothendieck ring of varieties is generated by the isomorphism classes~$[X]$ of algebraic varieties over~$k$, modulo the relations~$[X]=[U]+[Z]$ for~$Z\hookrightarrow X$ a closed subvariety, and~$U=X\setminus Z$ its complement. The product of varieties induces the multiplicative structure, such that~$[\pt]$ is the unit. An important element of this ring is the class of the affine line~$\LL=[\bA^1]=[\bP^1]-[\pt]$\index{L@$\LL$}, also called the Lefschetz class.

An alternative presentation for this ring in characteristic~0, relevant to our goal, is the Bittner presentation from \cite{MR2059227}. It says that~$\Kzero(\Var/k)$ is isomorphic to the ring generated by isomorphism classes~$[X]$ of smooth and proper algebraic varieties over~$k$, modulo the relations~$[\Bl_ZX]-[E]=[X]-[Z]$ for~$Z\hookrightarrow X$ a smooth closed subvariety, and~$E\to Z$ the exceptional divisor in the blowup~$\Bl_ZX\to X$.

Via the cut-and-paste relations one can obtain the standard identities
\begin{equation}
  \label{equation:K_0-fibration}
  [X]=[F][Y]
\end{equation}
if~$X\to Y$ is a Zariski-locally trivial fibration with fibers~$F$, and
\begin{equation}
  [\bP^n]=\sum_{i=0}^n\LL^i=\frac{1-\LL^{n+1}}{1-\LL}.
\end{equation}

In order to use this alternative presentation, and because we depend on other results in the literature which are only phrased for algebraically closed fields (but likely hold more generally), we will work over an arbitrary algebraically closed field~$k$ of characteristic~0.

\paragraph{Motivic measures}
A motivic measure is a ring morphism whose domain is the Grothen\-dieck ring of varieties. In our setting we consider the motivic measure
\begin{equation}
  \mu:\Kzero(\Var/k)\to\Kzero(\dgCat/k)
\end{equation}
from \cite[\S8]{MR2051435}, obtained by using the Bittner presentation and Orlov's blowup formula, and sending~$[X]$ to the class of the (unique) dg enhancement of~$\Kzero(\dgCat/k)$. Then the identity~\eqref{equation:class-in-K_0} matches up with \cref{conjecture:ur-sod}.

The Grothendieck ring of dg~categories~$\Kzero(\dgCat_k)$\index{K@$\Kzero(\dgCat_k)$} precisely encodes semiorthogonal decompositions, as it is generated by the quasiequivalence classes~$[\cC]$ of smooth and proper pretriangulated dg~categories, modulo the relations
\begin{equation}
  [\cC]=[\cA]+[\cB]
\end{equation}
for every semiorthogonal decomposition~$\cC=\langle\cA,\cB\rangle$.

Hence the image of the equality~\eqref{equation:class-in-K_0} is consistent with \cref{conjecture:ur-sod} (up to~2-torsion, as~$\LL\mapsto 1$), as explained by the following corollary.
\begin{corollary}
  \label{corollary:kzero-categories}
  We have the equality
  \begin{equation}
    \label{equation:class-in-dg-K_0}
    [\dbcoh{\odd}]=[\dbcoh{\Sym^{g-1}C}] + \sum_{i=0}^{g-2}2[\dbcoh{\Sym^iC}] + T'
  \end{equation}
  in~$\Kzero(\dgCat/k)$, for some class~$T'$ such that~$2\cdot T'=0$.
\end{corollary}
Again we expect that~$T'=0$, but our method of proof is not strong enough to remove this error term.

This is precisely the equality induced by the semiorthogonal decomposition from \cref{conjecture:ur-sod}, and therefore can be seen as further evidence for it.

\paragraph{Comparison to other results}
Motivic and cohomological invariants of~$\odd$ have been an active topic of interest for a long time, and many tools are used for this.
%The starting point for the computation of these invariants can be found in Newstead's \cite{MR0232015}, where the Betti numbers were first computed.
From the identity~\eqref{equation:class-in-K_0} it is possible to deduce known results on certain invariants of~$\odd$ by taking the appropriate motivic measures. Examples of this are given by
\begin{itemize}
  \item the Betti polynomial (with values in~$\bZ[t]$) \cite{MR0232015},
  \item the Hodge--Poincaré polynomial (with values in~$\bZ[x,y]$) \cite[Corollary~5.1]{MR1817504}.
\end{itemize}
It would be interesting to weaken the assumptions on the field~$k$ for \eqref{equation:class-in-K_0}, to also recover point-counting realizations \cite[Corollary~5.4]{MR1817504}.

With a view towards providing evidence for \cref{conjecture:ur-sod}, Lee has given in \cite[Theorem~1.2]{1806.11101} an isomorphism similar to~\eqref{equation:class-in-K_0} which holds in any \emph{semisimple} category of motives, such as the category of numerical motives. It is based on the isomorphism \cite[Theorem~2.7]{MR1895918} due to del Ba\~no, which also require this semisimplicity.

This result is stronger in the sense that it is known that the error term~$T$ vanishes. But it only works for (the Grothendieck group of) a semisimple category of motives. The motivic measure
\begin{equation}
  \mu\colon\Kzero(\Var/k)\to\Kzero(\Chow/k)
\end{equation}
which sends the class~$[X]$ of a smooth projective variety to the class~$[\mathfrak{h}(X)]$ of its Chow motive, allows one to obtain a strengthening (modulo vanishing of~$T$) of \cite[Theorem~1.2]{1806.11101}. See also \cite[Remark~2.8]{MR1895918}.

\subsection{The Thaddeus picture}
We break the proof of \cref{theorem:kzero} up into several steps. The main idea of the proof is to use Thaddeus' variation of GIT for moduli of stable pairs and compare the class of the variety when we cross a wall using a telescopic sum. This is the ``Thaddeus picture'' alluded to in the title of this section, and will be given in~\eqref{equation:thaddeus-picture}.

\paragraph{Setup}
Let~$d$ be an odd integer which is greater than~$2g-2$. Let~$\cL$ be a line bundle of degree $d$ and let $M^d_0=\bP(\HH^1(C,\cL^\vee))$. For each~$i \in \{1,\dots, (d-1)/2\}$, Thaddeus constructs in \cite{MR1273268} two smooth projective varieties~$M^d_i$ and~$\widetilde{M}^d_i$ such that~$\widetilde{M}_i^d$ is:
\begin{itemize}
  \item a blowup of~$M_{i-1}^d$ along a projective bundle~~$\bP(\cW_{i}^{-})$ over~$\Sym^iC$.
  \item a blowup of~$M_i^d$ along a projective bundle~$\bP(\cW_{i}^{+})$ over~$\Sym^iC$;
\end{itemize}
Here~$\cW_i^{-}$ (respectively~$\cW_i^{+}$) is vector bundle on~$\Sym^iC$ of rank~$i$ (respectively of rank $d+g-2i-1$).

This leads us to the following flip diagram:
\begin{equation}
  \label{equation:thaddeus-flip}
  \begin{tikzcd}[column sep=small, row sep=small]
    & & E \arrow[d, hook] \arrow[rrdd] \arrow[lldd] \\
    & & \widetilde{M}_i^d \arrow[ld] \arrow[rd] \\
    \bP(\cW_i^-) \arrow[r, hook] \arrow[rrdd] & M_{i-1}^d & & M_i^d & \bP(\cW_i^+) \arrow[l, hook'] \arrow[lldd] \\
    \\
    & & \Sym^iC
  \end{tikzcd}
\end{equation}

Summarizing the situation from~\cite{MR1273268} for all~$i$ we obtain the following theorem.
\begin{theorem}[Thaddeus]
  \label{theorem:thaddeus-picture}
  With the above notation:
  \begin{enumerate}
    \item There is a flip~$M^d_{i-1} \dashrightarrow M^d_{i}$ with center~$\Sym^iC$ and type~$(i,d+g-2i-1)$ for each~$i \in \{1,\dots, (d-1)/2\}$.
    \item We have that~$M_0^d\cong\bP^{d+g-2}$, and~$\widetilde{M}_1\cong M_1^d$.
    \item There is a natural map
      \begin{equation}
        \pi: M^d_{(d-1)/2}\rightarrow \odd
      \end{equation}
      with fiber~$\bP(\HH^0(C,\cE))$ over a stable bundle~$\cE$ in~$\odd$. Moreover, if~$d\geq 4g-3$, then~$\pi$ is a projective bundle associated to a vector bundle of rank~$d+2(1-g)$.
  \end{enumerate}
\end{theorem}

This yields the following picture, which we refer to as the ``Thaddeus picture'':
\begin{equation}
  \label{equation:thaddeus-picture}
  \begin{tikzcd}[column sep=small, row sep=small]
    & \widetilde{M}_1^d \arrow[ld] \arrow[rd, equals] & & \widetilde{M}_2^d \arrow[ld] \arrow[rd] & & & & \widetilde{M}_{(d-1)/2}^d \arrow[ld] \arrow[rd] \\
    M_0^d & & M_1^d & & M_2^d & \ldots & M_{(d-3)/2}^d & & M_{(d-1)/2}^d \arrow[d, "\pi"] \\
    & & & & & & & & \mathclap{\odd}
  \end{tikzcd}
\end{equation}
All morphisms in the diagram, except~$\pi$, are blowups.

To minimize the dimension of the moduli spaces and the number of steps involved in the construction, we will take~$d=4g-3$. Then
\begin{enumerate}
  \item $\dim M_i=5g-5$;
  \item we are considering moduli spaces~$M_0^{4g-3},\ldots,M_{2g-2}^{4g-3}$;
  \item the morphism~$M_{2g-2}$ is a~$\bP^{2g-2}$\dash fibration.
\end{enumerate}

\paragraph{Setup for the proof}
We now give some easy lemmas, as a setup for the proof in \cref{subsection:K_0-decomposition}. The following lemma follows directly from \cref{theorem:thaddeus-picture} and~\eqref{equation:K_0-fibration}.
\begin{lemma}
  \label{lemma:class-comparison}
  Let~$d$ be an odd integer greater than~$4g-4$, then we have that
  \begin{equation}
    [M_{(d-1)/2}^d]=\frac{1-\LL^{d+2(1-g)}}{1-\LL}[\odd]
  \end{equation}
  in~$\Kzero(\Var/k)$.
\end{lemma}

Using the description from \cref{lemma:class-comparison} for~$d=4g-3$ and~$d=4g-1$ and computing the difference gives the following description. One could take a less optimal choice, at the cost of obtaining a larger coefficient in \cref{theorem:kzero}.
\begin{lemma}
  \label{lemma:difference-comparison}
  We have that
  \begin{equation}
    (1+\LL)[\odd]=[M^{4g-1}_{2g-1}]-\LL^2[M^{4g-3}_{2g-2}]
  \end{equation}
  in~$\Kzero(\Var/k)$.
\end{lemma}

For every flip diagram as in~\eqref{equation:thaddeus-picture} one can prove the following, where only standard bookkeeping techniques are required.
\begin{proposition}
  \label{prop:d=4g-3}
  Let~$d$ be an odd integer greater than~$4g-4$. Then for~$i=1,\ldots,(d-1)/2$ the difference of the classes~$[M^d_{i}]$ and~$[M^d_{i-1}]$ satisfies
  \begin{equation}
    [M^d_{i}]-[M^d_{i-1}]=\frac{\LL}{(1-\LL)^2}\bigg((1-\LL^{d+g-2i-2})(1-\LL^i)-(1-\LL^{i-1})(1-\LL^{d+g-2i-1})\bigg)[\Sym^iC]
  \end{equation}
  in~$\Kzero(\Var/k)$.
\end{proposition}

\begin{proof}[Sketch of proof]
  By applying the blowup formula and the projective bundle formula twice with the appropriate codimension and relative dimension we get the equalities
  \begin{equation}
    \begin{aligned}
      [\widetilde{M}_i^d]&=[M_i^d]-[\bP(\cW_i^+)]+\frac{1-\LL^i}{1-\LL}[\bP(\cW_i^+)] \\
      &=[M_i^d]+\frac{\LL(1-\LL^{i-1})}{1-\LL}\frac{1-\LL^{d+g-2i-2}}{1-\LL}[\Sym^iC] \\
      [\widetilde{M}_i^d]&=[M_{i-1}^d]-[\bP(\cW_i^-)]+\frac{1-\LL^{d+g-2i-1}}{1-\LL}[\bP(\cW_i^-)] \\
      &=[M_{i-1}^d]+\frac{\LL(1-\LL^{d+g-2i-2})}{1-\LL}\frac{1-\LL^i}{1-\LL}[\Sym^iC]
    \end{aligned}
  \end{equation}
  and subtracting them gives the result.
\end{proof}

Finally, let us recall the following fundamental identity encoding the behavior of Abel--Jacobi morphisms, as e.g.~discussed in \cite[\S3]{MR2377891}.

%As the proof of \cref{theorem:kzero} depends on this identity, and uses similar (but easier) book-keeping techniques, we will outline how it goes.

\begin{proposition}
  \label{proposition:fundamental}
  Let~$C$ be a curve of genus~$g$. Let~$e\geq 0$ and denote~$a=(g-1)+e$. Then we have an equality
  \begin{equation}
    [\Pic^a C][\bP^{e-1}]=[\Sym^{g-1+e}C] - \LL^e[\Sym^{g-1-e}C]
  \end{equation}
  in~$\Kzero(\Var/k)$.
\end{proposition}
Observe that under our assumptions we have isomorphisms~$\Pic^iC\cong\Jac C$ for all~$i\in\bZ$, by the existence of a rational point.% It would be interesting to extend the analysis in the next section to more general situations, but we do not pursue this further.

\subsection{Proof of the decomposition}
\label{subsection:K_0-decomposition}
We will restrict ourselves to the cases where~$d=4g-3$ and~$4g-1$, i.e.~the first two degrees for which the morphism~$\pi$ in~\eqref{equation:thaddeus-picture} is an equidimensional projective fibration. We will compare different stages of the Thaddeus picture, using the following notation for the difference of classes of moduli of stable pairs, where~$i=0,\ldots,2g-2$.
\begin{equation}
  \delta M_i:=[M_{i}^{4g-1}]-\LL^2[M_i^{4g-3}].
\end{equation}
For notational convenience, we also set~$\delta M^{}_{-1}:=0$

Now for~$i=0,\ldots,2g-2$, we define
\begin{equation}
  X_i:=\delta M_i-\delta M_{i-1}.
\end{equation}

We can describe the classes~$X_i$ in the following way.
\begin{proposition}
  \label{proposition:middle-equation}
  For~$i=0,\ldots,2g-2$ we have that
  \begin{equation}
    X_i=\LL^i(1+\LL)[\Sym^{i}C]
  \end{equation}
\end{proposition}

\begin{proof}
  The proof follows from the definition of~$X_i$ and by applying \cref{prop:d=4g-3}, via the following chain of equalities.
  \begin{equation}
    \begin{aligned}
      X_{i}
      &=\delta M_{i}-\delta M_{i-1} \\
      &=[M^{4g-1}_{i}]-\LL^2[M^{4g-3}_{i}]-[M^{4g-1}_{i-1}]+\LL^2[M^{4g-3}_{i-1}] \\
      &=([M^{4g-1}_{i}]-[M^{4g-1}_{i-1}])-\LL^2([M^{4g-3}_{i}]-[M^{4g-3}_{i-1}]) \\
    \end{aligned}
  \end{equation}
  We now use \cref{prop:d=4g-3} and rewrite $X_{i}$ as
  \begin{equation}
    \begin{aligned}
      X_{i}
      &=\frac{\LL}{(1-\LL)^2}\bigg( \left( (1-\LL^{5g-3-2i})(1-\LL^{i})-(1-\LL^{5g-2-2i})(1-\LL^{i-1}) \right) \\
      &\qquad -\LL^2\left( (1-\LL^{5g-5-2i})(1-\LL^{i})-(1-\LL^{5g-4-2i})(1-\LL^{i-1})\right) \bigg)[\Sym^{i}C] \\
      %&=\frac{\LL}{(1-\LL)^2}\bigg( \left( (1-\LL^{5g-3-2i})-\LL^2(1-\LL^{5g-5-2i}) \right)(1-\LL^{i}) \\
      %&\qquad -\left( (1-\LL^{5g-2-2i})-\LL^2(1-\LL^{5g-4-2i}) \right)(1-\LL^{i-1}) \bigg)[\Sym^{i}C]\\
      &=\bigg(\frac{\LL(1-\LL^{i})(1-\LL^2)}{(1-\LL)^2}-\frac{\LL(1-\LL^2)(1-\LL^{i-1})}{(1-\LL)^2}\bigg)[\Sym^{i}C]\\
      &=\bigg(\frac{\LL(1-\LL^2)}{(1-\LL)^2}(\LL^{i-1}-\LL^i)\bigg)[\Sym^{i}C]\\
      &=\LL^{i}(1+\LL)[\Sym^{i}C].
    \end{aligned}
  \end{equation}
\end{proof}

Using \cref{proposition:middle-equation,proposition:fundamental} we obtain the following result, where we use the following polynomial in~$\LL$:
\begin{equation}
  \cP(i):=\LL^{2g-2-i}(1+\LL)(1+\LL+\ldots+\LL^{g-2-i}).
\end{equation}

\begin{proposition}
  \label{proposition:main-recursion}
  For~$i=0,\ldots,g-2$, we have that
  \begin{equation}
    X_i+X_{2g-2-i}=[\Sym^iC](\LL^i+\LL^{3g-3-2i})(1+\LL)+\cP(i)[\Jac(C)],
  \end{equation}
\end{proposition}

\begin{proof}
  %e=g-1-i$
  By \cref{proposition:fundamental} we get
  \begin{equation}
    \label{equation:kapranov-strong}
    [\Sym^{2g-2-i}C]=\LL^{g-i-1}[\Sym^iC]+[\bP^{g-2-i}][\Jac(C)]
  \end{equation}
  We apply \cref{proposition:middle-equation} for $i\in \{0,1,\dots, g-2\}$ to~\eqref{equation:kapranov-strong} to get
  \begin{equation}
    \begin{aligned}
      X_{2g-2-i}
      &=\LL^{2g-2-i}(1+\LL)[\Sym^{2g-2-i}C] \\
      &=\LL^{3g-3-2i}(1+\LL)[\Sym^iC]+\LL^{2g-2-i}(1+\LL)[\bP^{g-2-i}][\Jac(C)] \\
      &=\LL^{3g-3-2i}(1+\LL)[\Sym^iC]+\cP(i)[\Jac(C)]
    \end{aligned}
  \end{equation}
  Now by \cref{proposition:middle-equation} for~$i\leq g-1$, we get~$X_i=\LL^i(1+\LL)[\Sym^iC].$ Thus the proof follows.
\end{proof}

The following proposition is an important step in the proof of \cref{theorem:kzero}. It shows that there are no contributions of the Jacobian of~$C$ to the class of~$\odd$.
\begin{proposition}
  \label{proposition:polynomial-vanishing}
  We have that
  \begin{equation}
    \bigg(\sum_{i=0}^{g-2}\cP(i)\bigg)[\Jac(C)]=[M_{2g-2}^{4g-1}]-[M_{2g-1}^{4g-1}]
  \end{equation}
  in~$\Kzero(\Var/k)$.
\end{proposition}

\begin{proof}
  First let us simplify the left hand side of the above equation. We have that
  \begin{equation}
    \begin{aligned}
      \sum_{i=0}^{g-2}\cP(i)
      &=\frac{1+\LL}{1-\LL}\sum_{i=0}^{g-2}\LL^{2g-2-i}(1-\LL^{g-1-i}) \\
      &=\frac{(1+\LL)}{1-\LL}\sum_{i=0}^{g-2}\left( \LL^{2g-2-i}-\LL^{3g-3-2i} \right) \\
      &=\frac{\LL^{g}(1+\LL)}{1-\LL}\sum_{i=0}^{g-2}(\LL^{i}-\LL^{2i+1}) \\
      &=\frac{\LL^{g}(1+\LL)}{1-\LL}\left( \frac{1-\LL^{g-1}}{1-\LL}-\frac{\LL(1-(\LL^2)^{g-1})}{1-\LL^2} \right) \\
      &=\frac{\LL^g}{(1-\LL)^2}\left( (1-\LL^{g-1})(1-\LL^g) \right)
    \end{aligned}
  \end{equation}
  We will be done if we can show that the above expression multiplied by the class of~$\Jac(C)$ is equal to~$[M_{2g-2}^{4g-1}]-[M_{2g-1}^{4g-1}]$. First by \cref{proposition:fundamental} where we take~$e=g$ we get
  \begin{equation}
    [\Sym^{2g-1}C]=[\bP^{g-1}][\Jac(C)].
  \end{equation}
  Combining this with \cref{prop:d=4g-3} we get
  \begin{equation}
    \label{equation:ratz}
    [M^{4g-1}_{2g-1}]-[M^{4g-1}_{2g-2}]=\frac{L}{(1-\LL)^3}(1-\LL^{g})\bigg((1-\LL^{g-1})(1-\LL^{2g-1})-(1-\LL^g)(1-\LL^{2g-2})\bigg)[\Jac(C)].
  \end{equation}
  Now we use that
  \begin{equation}
    \label{equation:cakewasgood}
    (1-\LL^{g-1})(1-\LL^{2g-1})-(1-\LL^g)(1-\LL^{2g-2})=-\LL^{g-1}(1-\LL)+\LL^{2g-1}(1-\LL)
  \end{equation}
  Thus from~\eqref{equation:cakewasgood} we get
  \begin{equation}
    [M^{4g-1}_{2g-1}]-[M^{4g-1}_{2g-2}]=-\frac{\LL(1-\LL^g)}{(1-\LL)^3}\LL^{g-1}(1-\LL^{g-1})(1-\LL)[\Jac(C)],
  \end{equation}
  and we are done.
\end{proof}

Now we are ready to complete the proof of \cref{theorem:kzero}.

\begin{proof}[Proof of \cref{theorem:kzero}]
  First we write the class of~$\odd$ in terms of the difference of classes of the smooth projective varieties considered by Thaddeus. We apply \cref{lemma:difference-comparison} to get
  \begin{equation}
    \label{equation:moduli-diff}
    (1+\LL)[\odd]=[\delta M_{2g-2}]+[M_{2g-1}^{4g-1}]-[M_{2g-2}^{4g-1}]
  \end{equation}
  Now we can write
  \begin{equation}
    [\delta M_{2g-2}]=\sum_{i=0}^{2g-2}X_i,
  \end{equation}
  where $X_i$'s are as used in \cref{proposition:main-recursion}. Thus applying \cref{proposition:main-recursion} we get
  \begin{equation}
    \begin{aligned}
      &(1+\LL)[\odd] \\
      &\quad=[\delta M_{2g-2}^{4g-3}]+[M_{2g-1}^{4g-1}]-[M_{2g-2}^{4g-1}] \\
      &\quad= \sum_{i=0}^{2g-2}X_i+[M_{2g-1}^{4g-1}]-[M_{2g-2}^{4g-1}] \\
      &\quad= \sum_{i=0}^{g-2}(X_i+X_{2g-2-i})+X_{g-1}+[M_{2g-1}^{4g-1}]-[M_{2g-2}^{4g-1}] \\
      &\quad=\LL^{g-1}(1+\LL)[\Sym^{g-1}C]+\sum_{i=0}^{g-2}(\LL^i+\LL^{3g-3-2i})(1+\LL)[\Sym^iC] \\
      &\qquad+\bigg(\sum_{i=0}^{g-2}\cP(i)\bigg)[\Jac(C)]+[M_{2g-1}^{4g-1}]-[M_{2g-2}^{4g-1}] \\
      &\quad=\LL^{g-1}(1+\LL)[\Sym^{g-1}C]+\sum_{i=0}^{g-2}(\LL^i+\LL^{3g-3-2i})(1+\LL)[\Sym^iC]
    \end{aligned}
  \end{equation}
  where the last step is by \cref{proposition:polynomial-vanishing}.
\end{proof}

\subsection{Motivic zeta functions and a Harder-type formula}
\label{subsection:motivic-measure}
As an application of the results above we can give an analogue of Harder's point counting formula from \cite{MR0262246} (see also \cite[Corollary~2.11]{MR1895918}). We do this by exhibiting an identity in the Grothendieck ring of varieties. Note however that the ground field cannot be chosen to be~$\bF_q$, hence Harder's formula cannot be obtained by applying the motivic measure
\begin{equation}
\#:\Kzero(\Var/\bF_q)\to\bZ:[X]\mapsto\#X(\bF_q).
\end{equation}

For any variety~$X$, we denote the~$n$th symmetric~$\Sym^nX$ to be~$X^n/\Sym_n$, where~$\Sym_n$ is the symmetric group of~$n$ letters. All the symmetric powers can be put together to give Kapranov's \emph{motivic zeta function}, as introduced in \cite[\S1.3]{math/0001005}:
\begin{equation}
  \ZZKap(X,t):=\sum_{n\geq 0}[\Sym^nX]t^n \in \Kzero(\Var/k)[[t]].
\end{equation}
This is a universal version of the Hasse--Weil zeta function, valid for arbitrary ground fields, where the counting measure for~$k=\bF_q$ gives the usual Hasse--Weil zeta function.

The following theorem is due to Kapranov \cite[Theorem~1.1.9]{math/0001005}, and shows how the motivic zeta function has properties similar to the usual Hasse--Weil zeta function.%\footnote{Observe that in \cite[Remarks~1.3.5(b)]{math/0001005} it is conjectured that the motivic zeta function is rational for arbitrary smooth projective varieties, but this has been shown to be false for~$k=\C$ in \cite{MR1996804}.}.
\begin{theorem}[Kapranov]
  \label{theorem:kapranov}
  Let~$C$ be a smooth curve. The motivic zeta function of~$C$ is a rational function of the following form:
  \begin{equation}
    \ZZKap(C,t)=\frac{F_{2g}(t)}{(1-t)(1-\LL t)},
  \end{equation}
  where~$F_{2g}(t)$ is a polynomial of degree~$2g$. Moreover~$\ZZKap(C,t)$ (resp.~$F_{2g}(t)$) satisfies the functional equation
  \begin{equation}
    \ZZKap(C,t)=\LL^{g-1}t^{2g-2}\ZZKap\left( C,\frac{1}{\LL t} \right)
  \end{equation}
  resp.
  \begin{equation}
    F_{2g}(t)=\LL^gt^{2g}F_{2g}\left( \frac{1}{\LL t} \right).
  \end{equation}
\end{theorem}

We can rearrange and reinterpret the terms in the motivic zeta function as follows.
\begin{proposition}
  \label{proposition:kapranov-reinterpreted}
  We have the identity
  \begin{equation}
    \label{equation:kapranov-reinterpreted}
    \begin{aligned}
      \ZZKap(C,\LL)
      &=\left( \sum_{i=0}^{g-2}[\Sym^iC](\LL^i+\LL^{3g-2i-3}) \right) +[\Sym^{g-1}C]\LL^{g-1} \\
      &\quad + \sum_{i=0}^{g-2}[\Jac(C)][\bP^{g-i-2}]\LL^{2g-2-i}+\frac{[\Jac(C)]\LL^{2g-1}}{1-\LL}\bigg(\frac{1}{1-\LL}-\frac{\LL^{g}}{(1-\LL^2)}\bigg)
    \end{aligned}
  \end{equation}
  in~$\Kzero(\Var/k)$.
\end{proposition}

\begin{proof}
  Using \cref{proposition:fundamental} we have the identity\footnote{One avoids manipulating infinite sums of~$\LL$ by observing that~$\ZZKap(C,t)$ is a rational function in~$t$, and only finitely many copies of powers of~$\LL$ are contributing to each power of~$t$ before evaluation.}
  \begin{equation}
    \begin{aligned}
      \ZZKap(C,\LL)
      &=\left( \sum_{i=0}^{g-2}[\Sym^i]\LL^i \right)
      +[\Sym^{g-1}C]\LL^{g-1}
      +\left( \sum_{i=0}^{g-2}[\Sym^{2g-i-2}C]\LL^{2g-i-2} \right) \\
      &\qquad+\left( \sum_{i\geq 2g-1}[\Sym^iC]\LL^i \right) \\
      &=\left( \sum_{i=0}^{g-2}[\Sym^i]\LL^i \right)
      +[\Sym^{g-1}C]\LL^{g-1}\\
      &\qquad+\left( \sum_{i=0}^{g-2}[\Sym^iC]\LL^{3g-2i-3}+[\Jac C][\bP^{g-i-2}]\LL^{2g-i-2} \right) \\
      &\qquad+\left( \sum_{i\geq 2g-1}[\Jac C][\bP^{i-g}]\LL^i \right).
    \end{aligned}
  \end{equation}
  It now suffices to see that
  \begin{equation}
    \left( \sum_{i\geq 2g-1}[\Jac C][\bP^{i-g}]\LL^i \right)
    =
    \frac{[\Jac C]\LL^{2g-1}}{1-\LL}\left( \frac{1}{1-\LL}-\frac{\LL^g}{1-\LL^2} \right)
  \end{equation}
  which follows from an immediate verification.
\end{proof}

The following lemma is also an immediate verification.
\begin{lemma}
  \label{lemma:L-identity}
  We have the identity
  \begin{equation}
    \sum_{i=0}^{g-2}\LL^{2g-i-2}(1-\LL^{g-i-1})(1-\LL^2)+\LL^{2g-1}(1+\LL-\LL^g)
    =
    \LL^g
  \end{equation}
  in~$\Kzero(\Var/k)$.
\end{lemma}

Using \cref{proposition:kapranov-reinterpreted,lemma:L-identity} we can then obtain the following corollary to \cref{theorem:kzero}. The right-hand side is an element of~$\Kzero(\Var/k)$ (and not some completion) by \cref{theorem:kapranov}, so that the first term of the right-hand side is in fact~$F_{2g}(\LL)$.
\begin{corollary}
  We have the identity
  \begin{equation}
    (1-\LL)(1-\LL^2)[\odd]=(1-\LL)(1-\LL^2)\ZZKap(C,\LL)-\LL^g[\Jac C]
  \end{equation}
  in~$\Kzero(\Var/k)$.
\end{corollary}

\begin{proof}
  It suffices to rewrite~$[\bP^{g-i-2}]$ as~$\frac{1-\LL^{g-i-1}}{1-\LL}$ and then multiply both sides of~\eqref{equation:kapranov-reinterpreted} with~$(1-\LL)(1-\LL^2)$, and apply \cref{lemma:L-identity}.
\end{proof}

This identity is then the analogue of
Harder's formula from \cite[Corollary~2.11]{MR1895918}
using the functional equation for the zeta function.
% It would be interesting to extend these results
% to more general base fields, and in particular
% recover the actual point counting formula,
% but this is outside the scope of this work.

%% file: decomp-critical.tex
\appendix

\section{Dimension of critical loci}
\label{section:critical}

We will now determine the dimensions of the different components of the critical loci,
which is part of \cref{theorem:critical-loci}.
The goal is the following proposition.

\begin{proposition}
  \label{proposition:dimension-critical-locus}
  Let~$\Gamma_{g,g-1}$ denote the necklace graph of genus~$g$,
  with the last vertex colored as in \cref{figure:necklace-graph}.
  Let~$\widetilde{W}_g$ denote the associated graph potential.
  Then the dimension of the critical locus of~$\widetilde{W}_g$
  over the critical value with absolute value~$8g-8-8k$ for~$k=0,\ldots,g-1$ is~$g-1$.
\end{proposition}

The proof of \cref{proposition:dimension-critical-locus} is a lengthy computation,
in order to set up an inductive description of the critical loci.

\begin{figure}[t!]
  \centering
  \begin{tikzpicture}[scale=1]
    \clip (-1,-1) rectangle (10,2);

    \node[vertex] (A) at (0,0) {};
    \node[vertex] (B) at (1,0) {};
    \node[vertex] (C) at (2,0) {};
    \node[vertex] (D) at (3,0) {};
    \node[vertex] (E) at (4,0) {};
    \node[vertex] (F) at (5,0) {};
    \node[vertex] (Fbis) at (6,0) {};
    \node[vertex] (Gbis) at (7,0) {};
    \node[vertex] (G) at (8,0) {};
    \node[vertex,fill] (H) at (9,0) {};

    \draw (A) edge [bend left]  node [above] {$x_1$} (B);
    \draw (A) edge [bend right] node [below] {$y_1$} (B);
    \draw (B) edge node [below] {$z_1$} (C);
    \draw (C) edge [bend left]  node [above] {$x_2$} (D);
    \draw (C) edge [bend right] node [below] {$y_2$} (D);
    \draw (D) edge node [below] {$z_2$} (E);
    \draw (E) edge [bend left]  node [above] {$x_3$} (F);
    \draw (E) edge [bend right] node [below] {$y_3$} (F);
    \draw (F) edge node [below] {$z_3$} (Fbis);
    \node at (6.5,0) {$\ldots$};
    \draw (Gbis) edge node [below] {$z_{g-1}$} (G);
    \draw (G) edge [bend left]  node [above] {$x_{g-1}$} (H);
    \draw (G) edge [bend right] node [below] {$y_{g-1}$} (H);
    \draw (A) edge [bend left=160] node [above] {$z_1=z_g$} (H);
  \end{tikzpicture}
  \caption{Labelling of variables on the necklace graph}
  \label{figure:necklace-graph}
\end{figure}
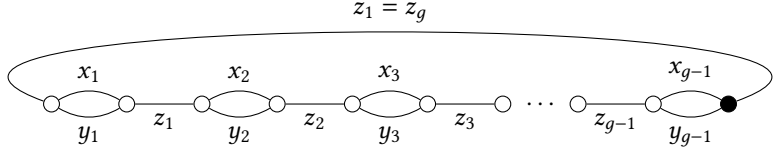

Instead of the coordinates~$x_i,y_i,z_i$ we will work with the coordinates
\begin{equation}
  \begin{aligned}
    u_i&:= x_iy_i \\
    v_i&:= x_i/y_i
  \end{aligned}
\end{equation}
which makes the graph potential easier to work with.

The necklace graph from \cref{figure:necklace-graph} can be decomposed in two ways,
using \emph{beads} and \emph{strings}.
In the coordinates~$u_i,v_i,z_i$ we describe the local pictures and the labelling of the variables in \cref{figure:local-picture-necklace-graph}.

We will moreover use the following notation for two functions on~$\Gm$:
\begin{equation}
  \begin{aligned}
    J^+(x)&:= x+x^{-1} \\
    J^-(x)&:= x-x^{-1},
  \end{aligned}
\end{equation}
so that~$x\frac{\partial}{\partial x}J^+(x)=J^-(x)$.% and~$x\frac{\partial}{\partial x}J^-(x)=J^+(x)$.

With this notation and the decomposition of the necklace graph the graph potential can be written as the sum of either bead potentials or string potentials. Here the $i$th \emph{bead potential} is given as
\begin{equation}
  \label{equation:bead-potential}
  \widetilde{W}_{g,i}^{\mathrm{b}}:=
  \begin{cases}
    z_iJ^+(u_i)+z_i^{-1}J^+(v_i)+z_{i+1}J^+(u_i)+z_{i+1}^{-1}J^+(v_i) & i=1,\ldots,g-2 \\
    z_{g-1}J^+(u_{g-1})+z_{g-1}^{-1}J^+(v_{g-1})+z_1J^+(v_{g-1})+z_1^{-1}J^+(u_{g-1}) & i=g-1
  \end{cases}
\end{equation}
and the $i$th \emph{string potential} is given as
\begin{equation}
  \label{equation:string-potential}
  \widetilde{W}_{g,i}^{\mathrm{s}}:=
  \begin{cases}
    z_1J^+(v_{g-1})+z_1^{-1}J^+(u_{g-1})+z_1J^+(u_1)+z_1^{-1}J^+(v_1) & i=1 \\
    z_iJ^+(u_{i-1})+z_i^{-1}J^+(v_{i-1})+z_iJ^+(u_i)+z_i^{-1}J^+(v_i) & i=2,\ldots,g-1 \\
  \end{cases}.
\end{equation}
This allows us to write
\begin{equation}
  \widetilde{W}_g
  =
  \sum_{i=1}^{g-1}\widetilde{W}_{g,i}^{\mathrm{b}}
  =
  \sum_{i=1}^{g-1}\widetilde{W}_{g,i}^{\mathrm{s}}.
\end{equation}

\begin{figure}[t!]
  \centering
  \begin{subfigure}[b]{.49\textwidth}
    \centering
    \begin{tikzpicture}[scale=1]
      \node[vertex] (first)  at (0,0) {};
      \node[vertex] (second) at (1,0) {};
      \node (A) at (-1,0) {};
      \node (B) at (2,0)  {};

      \draw (A)     edge node [above] {$z_i$} (first);
      \draw (first) edge [bend left]  node [above] {$u_i$} (second);
      \draw (first) edge [bend right] node [below] {$v_i$} (second);
      \draw (B)     edge node [above] {$z_{i+1}$} (second);
    \end{tikzpicture}
    \caption{Local picture of $i$th bead}
    \label{subfigure:bead}
  \end{subfigure}
  \begin{subfigure}[b]{.49\textwidth}
    \centering
    \begin{tikzpicture}[scale=1]
      \node[vertex]      (first)  at (0,0) {};
      \node[vertex,fill] (second) at (1,0) {};
      \node (A) at (-1,0) {};
      \node (B) at (2,0)  {};

      \draw (A)     edge node [above] {$z_{g-1}$} (first);
      \draw (first) edge [bend left]  node [above] {$u_{g-1}$} (second);
      \draw (first) edge [bend right] node [below] {$v_{g-1}$} (second);
      \draw (B)     edge node [above] {$z_g=z_1$} (second);
    \end{tikzpicture}
    \caption{Local picture of the final bead}
    \label{subfigure:final-bead}
  \end{subfigure}

  \begin{subfigure}[b]{.49\textwidth}
    \centering
    \begin{tikzpicture}[scale=1]
      \node[vertex] (first)  at (0,0) {};
      \node[vertex] (second) at (1,0) {};
      \node (A) at (-1,0.25)  {};
      \node (B) at (-1,-0.25) {};
      \node (C) at (2,0.25)   {};
      \node (D) at (2,-0.25)  {};

      \draw (A)     edge [bend left]  node [above] {$u_{i-1}$} (first);
      \draw (B)     edge [bend right] node [below] {$v_{i-1}$} (first);
      \draw (first) edge node [above] {$z_i$} (second);
      \draw (C)     edge [bend right] node [above] {$u_i$} (second);
      \draw (D)     edge [bend left]  node [below] {$v_i$} (second);
    \end{tikzpicture}
    \caption{Local picture of the $i$th string}
    \label{subfigure:string}
  \end{subfigure}
  \begin{subfigure}[b]{.49\textwidth}
    \centering
    \begin{tikzpicture}[scale=1]
      \node[vertex]      (first)  at (0,0) {};
      \node[vertex,fill] (second) at (4,0) {};
      \node (A) at (1,0.25)  {};
      \node (B) at (1,-0.25) {};
      \node (C) at (3,0.25)   {};
      \node (D) at (3,-0.25)  {};

      \draw (A)     edge [bend right]    node [above] {$u_{1}$} (first);
      \draw (B)     edge [bend left]     node [below] {$v_{1}$} (first);
      \draw (first) edge [bend left=120] node [above] {$z_1=z_g$} (second);
      \draw (C)     edge [bend left]     node [above] {$u_{g-1}$} (second);
      \draw (D)     edge [bend right]    node [below] {$v_{g-1}$} (second);
    \end{tikzpicture}
    \caption{Local picture of the long string}
    \label{subfigure:long-string}
  \end{subfigure}

  \caption{Local pictures of beads and strings in the necklace graph of genus $g$}
  \label{figure:local-picture-necklace-graph}
\end{figure}
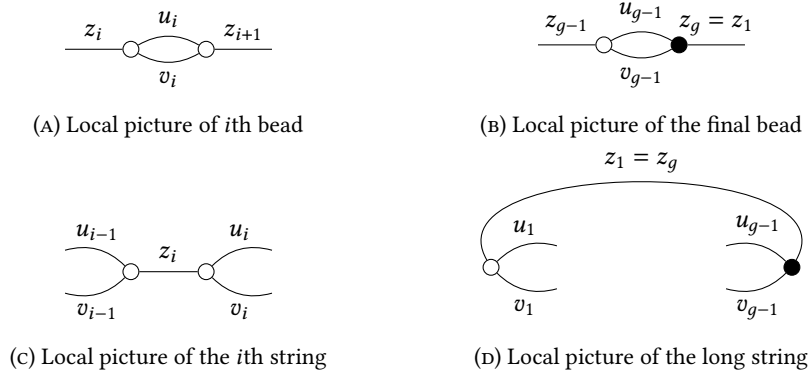

In the description of the critical locus we therefore have the equations
\begin{equation}
  \label{equation:ui-vi-condition}
  \begin{aligned}
    u_i\frac{\partial}{\partial u_i}\widetilde{W}_g
    &=u_i\frac{\partial}{\partial u_i}\widetilde{W}_{g,i}^{\mathrm{b}}
    =J^-(u_i)(z_i+z_{i+1})
    =0 \\
    v_i\frac{\partial}{\partial v_i}\widetilde{W}_g
    &=v_i\frac{\partial}{\partial v_i}\widetilde{W}_{g,i}^{\mathrm{b}}
    =J^-(v_i)(z_i+z_{i+1})
    =0
  \end{aligned}
\end{equation}
for~$i=1,\ldots,g-2$, and
\begin{equation}
  \label{equation:ug-vg-condition}
  \begin{aligned}
    u_{g-1}\frac{\partial}{\partial u_{g-1}}\widetilde{W}_g
    &=u_{g-1}\frac{\partial}{\partial u_{g-1}}\widetilde{W}_{g,g-1}^{\mathrm{b}}
    =J^-(u_{g-1})(z_{g-1}+z_1^{-1})
    = 0\\
    v_{g-1}\frac{\partial}{\partial v_{g-1}}\widetilde{W}_g
    &=v_{g-1}\frac{\partial}{\partial v_{g-1}}\widetilde{W}_{g,g-1}^{\mathrm{b}}
    =J^-(v_{g-1})(z_{g-1}^{-1}+z_1)
    =0.
  \end{aligned}
\end{equation}
Likewise we obtain the conditions
\begin{equation}
  \label{equation:zi-condition}
  z_i\frac{\partial}{\partial z_i}\widetilde{W}_g
  =z_i\frac{\partial}{\partial z_i}\widetilde{W}_{g,i}^{\mathrm{s}}
  =z_i(J^+(u_{i-1})+J^+(u_i))-z_i^{-1}(J^+(v_{i-1})+J^+(v_i))
  =0
\end{equation}
for~$i=2,\ldots,g-1$, and
\begin{equation}
  \label{equation:z1-condition}
  z_1\frac{\partial}{\partial z_i}\widetilde{W}_g
  =z_1\frac{\partial}{\partial z_1}\widetilde{W}_{g,1}^{\mathrm{s}}
  =z_1(J^+(u_1)+J^+(v_{g-1}))-z_1^{-1}(J^+(v_1)+J^+(u_{g-1}))
  =0
\end{equation}

We need to analyse the system of equations given by \eqref{equation:ui-vi-condition}, \eqref{equation:ug-vg-condition}, \eqref{equation:zi-condition} and \eqref{equation:z1-condition},
so that we can describe its solutions
and compute the value of the graph potential at these solutions.

The \emph{expected dimension} of the connected component of the critical locus
with critical value given by~$8(g-1-k)\sqrt{-1}^{1+(-1)^{k+1}}$ (for~$k=0,\ldots,2g-2)$)
is~$k$ for~$k=0,\ldots,g-1$, and~$g-k+1$ for~$k=g,\ldots,2g-2$.
I.e.~for~$k=0,2g-2$ we expect isolated critical points,
whilst for~$k=g-1$ we obtain critical value~0 and expect a critical locus of dimension~$g-1$.

Under the assumptions \eqref{equation:zi-condition} and \eqref{equation:z1-condition} it is possible to rewrite the string potentials by removing the~$z_i^{-1}$,
so that the graph potential can be expressed as
\begin{equation}
  \label{equation:simplified-potential}
  2\left(
    z_1(J^+(u_1)+J^+(v_{g-1})
    \sum_{i=2}^{g-1}z_i(J^+(u_{i-1})+J^+(u_i))
  \right).
\end{equation}
We will get rid of the factor~2, so that we are interesting in the values~$4(g-1-k)\sqrt{-1}^{1+(-1)^{k+1}}$ for the resulting expression.

\paragraph{The case $u_i^2=v_i^2=1$}
As a starting point we consider the situation in which the conditions \eqref{equation:ui-vi-condition} and \eqref{equation:ug-vg-condition} are satisfied
by ensuring that~$J^-(u_i)=J^-(v_i)=0$ for all~$i=1,\ldots,g-1$.
We need to analyse the condition on the remaining variables~$z_i$,
and what the resulting critical loci and critical values are.

The condition on the~$u_i$ and~$v_i$ ensures that $u_i=\pm1$ and~$v_i=\pm1$ for all~$i=1,\ldots,g-1$.
We will encode the sign choice the variables of the~$i$th string potential as a matrix,
so that~$(\begin{smallmatrix} s_{i-1}^u & s_i^u \\ s_{i-1}^v & s_i^v \end{smallmatrix})$
resp.~$(\begin{smallmatrix} s_{1}^u & s_{g-1}^u \\ s_{1}^v & s_{g-1}^v \end{smallmatrix})$
denotes the sign in the choice of the~$u_i$ and~$v_i$.

\begin{lemma}
  \label{lemma:ui-vi-pm-1}
  Let~$u_i^2=v_i^2=1$ for all~$i=1,\ldots,g-1$.
  If for some string the sign choice is inadmissible,
  i.e.~the parity of the signs in the sign matrix is odd,
  then at least one of the equations \eqref{equation:zi-condition} and \eqref{equation:z1-condition} cannot be satisfied.

  Assume that for every string the sign choice is admissible (i.e.~the parity is even).
  Then we have that
  \begin{enumerate}
    \item the variables~$z_i$ are either free, or are necessarily equal to~$\pm1,\pm\sqrt{-1}$;
    \item the critical loci have the expected dimension.
  \end{enumerate}
\end{lemma}

\begin{proof}
  Observe that~$J^+(\pm1)=\pm2$,
  so that the coefficients of the~$z_i^{\pm1}$ in \eqref{equation:zi-condition}, \eqref{equation:z1-condition} are~$-4,0,4$.
  If the sign choice is inadmissible for the~$i$th string,
  then in \eqref{equation:zi-condition} exactly one of the coefficients of~$z_i^{\pm1}$ vanishes.
  but then the system has no solution in~$(\mathbb{C}^\times)^{3g-3}$.

  If the sign choice is admissible we can analyse the system as in \cref{table:sign-choices}: we describe~$z_i\frac{\partial}{\partial z_i}\widetilde{W}_g$ and deduce the condition it imposes on~$z_i$. This proves the first part.
  %We will write the sign choice for the first (resp.~$i$th) string as~$(\begin{smallmatrix} u_1 & u_{g-1} \\ v_1 & v_{g-1} \end{smallmatrix})$ resp.~$(\begin{smallmatrix} u_{i-1} & u_i \\ v_{i-1} & v_i \end{smallmatrix})$.
  %Then \eqref{equation:zi-condition}, \eqref{equation:z1-condition} and the ensuing condition on the~$z_i$ are as in \cref{table:sign-choices}.

  It remains to understand the dimension of the critical loci.
  For this we consider the expression of the necklace graph potential as the sum of string potentials.
  The expression in \eqref{equation:zi-condition}, \eqref{equation:z1-condition} is equal to \eqref{equation:string-potential} up to a sign.

  Observe that the parity condition for consecutive strings force the choices of~$z_i$ to be either all real, or all imaginary.
  It follows that the critical loci have the expected dimension by the evaluation of the string potentials as given in \cref{table:sign-choices}.
\end{proof}

\begin{table}
  \centering
  \begin{tabular}{c|ccc|ccc}
    & $z_1\frac{\partial}{\partial z_i}\widetilde{W}_g$ & condition on $z_1$ & $\widetilde{W}_{g,1}^{\mathrm{s}}$ & $z_i\frac{\partial}{\partial z_i}\widetilde{W}_g$ & condition on $z_i$ & $\widetilde{W}_{g,i}^{\mathrm{s}}$ \\
    \midrule
    $(\begin{smallmatrix} + & + \\ + & + \end{smallmatrix})$ & $4J^-(z_1)$  & $z_1=\pm1$         &  & $4J^-(z_i)$  & $z_i=\pm1$         & \\
    $(\begin{smallmatrix} - & - \\ - & - \end{smallmatrix})$ & $-4J^-(z_1)$ & $z_1=\pm1$         &  & $-4J^-(z_i)$ & $z_i=\pm1$         & \\
    $(\begin{smallmatrix} + & + \\ - & - \end{smallmatrix})$ & $0$          & $z_1$ free         &  & $4J^+(z_i)$  & $z_i=\pm\sqrt{-1}$ & \\
    $(\begin{smallmatrix} - & - \\ + & + \end{smallmatrix})$ & $0$          & $z_1$ free         &  & $-4J^+(z_i)$ & $z_i=\pm\sqrt{-1}$ & \\
    $(\begin{smallmatrix} + & - \\ - & + \end{smallmatrix})$ & $4J^+(z_1)$  & $z_1=\pm\sqrt{-1}$ &  & $0$          & $z_i$ free         & \\
    $(\begin{smallmatrix} - & + \\ + & - \end{smallmatrix})$ & $-4J^+(z_1)$ & $z_1=\pm\sqrt{-1}$ &  & $0$          & $z_i$ free         & \\
    $(\begin{smallmatrix} + & - \\ + & - \end{smallmatrix})$ & $0$          & $z_1$ free         &  & $0$          & $z_i$ free         & \\
    $(\begin{smallmatrix} - & + \\ - & + \end{smallmatrix})$ & $0$          & $z_1$ free         &  & $0$          & $z_i$ free         & \\
  \end{tabular}
  \caption{Admissible sign choices and conditions \eqref{equation:zi-condition}, \eqref{equation:z1-condition}}
  \label{table:sign-choices}
\end{table}

\paragraph{Inductive description}
We now consider the case where there is at least one~$u_i^2\neq 1$ or~$v_i^2\neq 1$.
The goal is to relate the equations of the critical locus for the necklace graph of genus~$g$ to that of the necklace graph of genus~$g-2$ and~$g-1$.

Hence we first need to describe what happens for~$g=2$ and~$g=3$.
\begin{lemma}
  \label{lemma:base-case}
  \Cref{proposition:dimension-critical-locus} holds for~$g=2,3$.
\end{lemma}

\begin{proof}
  For~$g=2$ we have~$\widetilde{W}_2=J^+(z)(J^+(u)+J^+(v))$
  and thus the system of partial derivatives
  \begin{equation}
    \left\{
      \begin{aligned}
        0&=\frac{\partial\widetilde{W}_2}{\partial u}=J^+(z)J^-(u) \\
        0&=\frac{\partial\widetilde{W}_2}{\partial v}=J^+(z)J^-(v) \\
        0&=\frac{\partial\widetilde{W}_2}{\partial z}=J^-(z)(J^+(u)+J^+(v))
      \end{aligned}
    \right.
  \end{equation}
  so that a case-by-case analysis yields
  \begin{itemize}
    \item $J^+(z)=0$ and $J^+(u)+J^+(v)=0$ corresponds to 1-dimensional critical loci with critical value~0;
    \item $J^-(z)=0$ and $J^-(u)=J^-(v)=0$ corresponds to 0-dimensional critical loci with critical value~$\pm8$.
  \end{itemize}
  For~$g=3$ is similar, using the reduction methods from the next lemmas,
  except for the case where~$u_1^2\neq 1$ and~$z_1+z_2^{-1}=0$
  where the reduction methods would yield~$g=1$.
  If~$u_1^2\neq 1$ then~$z_1=-z_2$ and~$z_1=-z_2^{-1}$ by assumption.
  Then~$z_2=\pm 1$ and~$z_1=\mp 1$.
  Now cancelling~$z_1$ and~$z_2$ gives
  \begin{equation}
    2J^+(u_1)+J^+(u_2)+J^+(v_2)=2J^+(v_1)+J^+(u_2)+J^+(v_2)
  \end{equation}
  which implies~$J(u_1)=J(v_1)$.
  Substituting this we get~$J^+(u_2)=J^+(v_2)$,
  and thus this critical locus has dimension~2 and critical value~0.
\end{proof}

We can assume that~$u_1^2\neq 1$, by the rational change of coordinates in the change of coloring \cite[\S2.2]{gp-tqft}.
From the condition \eqref{equation:ui-vi-condition} for~$u_1$ we get
\begin{equation}
  \label{equation:z1-z2}
  z_1+z_2=0,
\end{equation}
and thus condition \eqref{equation:ui-vi-condition} for~$v_1$ is always satisfied.

We will first analyse the case where~$z_2+z_3=0$ or~$z_1+z_{g-1}^{-1}=0$.
Then we will be done once we've dealt with the case~$z_1+z_{g-1}^{-1}\neq 0$ and~$z_2+z_3\neq 0$.

\begin{lemma}
  Let~$u_1^2\neq 1$. Assume that~$z_2+z_3=0$. Then the critical loci have the expected dimension.
\end{lemma}

\begin{proof}
  By \eqref{equation:z1-z2} and the assumption we have that~$z_1=z_3$. The equations \eqref{equation:ui-vi-condition} for~$i=1,2$ are automatically satisfied and can be ignored.

  We can rewrite \eqref{equation:zi-condition} for~$i=2$ as
  \begin{equation}
    -z_3(J^+(u_1)+J^+(u_2))-z_3^{-1}(J^+(v_1)+J^+(v_2))=0
  \end{equation}
  and \eqref{equation:z1-condition} as
  \begin{equation}
    z_3(J^+(u_3)+J^+(v_{g-1}))-z_3^{-1}(J^+(v_1)+J^+(u_{g-1}))=0.
  \end{equation}
  By summing these equations with \eqref{equation:zi-condition} for~$i=3$ we obtain
  \begin{equation}
    z_3(J^+(u_3)+J^+(v_{g-1}))-z_3^{-1}(J^+(v_3)+J^+(u_{g-1}))=0.
  \end{equation}
  Finally observe that the first few terms in the expression of the graph potential from \eqref{equation:simplified-potential} read
  \begin{equation}
    \widetilde{W}_g=
    z_1(J^+(u_1)+J^+(v_{g-1}))
    +z_2(J^+(u_1)+J^+(u_2))
    +z_3(J^+(u_3)+J^+(u_3))
    +\ldots
  \end{equation}
  so that by cancellations from the equalities~$z_1=-z_2=z_3$ it reduces to that of genus~$g-2$,
  and the sought-after critical values are of the form~$4((g-2)-1-(k-2))\sqrt{-1}^{1+(-1)^{k+1}}$ for~$k=1,\ldots,2g-3$.

  Hence we have reduced the system of equations and the graph potential for the necklace graph of genus~$g$ to that of genus~$g-2$.
  By the base case from \cref{lemma:base-case},
  the fact that the critical loci for the necklace graph of genus~$g-2$ are of the expected dimension,
  and that we used the two equations~$z_1+z_2=0$ and~$z_2+z_3=0$,
  we obtain that the critical loci for~$k=1,\ldots,2g-3$ are also of the expected dimension~$k+2$.
\end{proof}

A similar analysis shows the following.
\begin{lemma}
  Let~$u_1^2\neq 1$. Assume that~$z_1+z_{g-1}^{-1}=0$. Then the critical loci have the expected dimension.
\end{lemma}

Finally we prove the following result.
\begin{lemma}
  Let~$u_1^2\neq 1$. Assume that~$z_1+z_{g-1}^{-1}\neq 0$ and~$z_2+z_3\neq 0$. Then the critical loci have the expected dimension.
\end{lemma}

\begin{proof}
  By \eqref{equation:ui-vi-condition} for~$i=2$ and \eqref{equation:ug-vg-condition} we obtain that~$u_2^2=v_2^2=u_{g-1}^2=v_{g-1}^2=1$.
  This brings us in a situation reminiscent of that of \cref{lemma:ui-vi-pm-1},
  but we will only consider the sign of~$(\begin{smallmatrix} u_2 & u_{g-1} \\ v_2 & v_{g-1} \end{smallmatrix})$.

  We can rewrite \eqref{equation:zi-condition} for~$i=2$ as
  \begin{equation}
    z_2^2=\frac{J^+(v_1)+J^+(v_2)}{J^+(u_1)+J^+(u_2)}
  \end{equation}
  and \eqref{equation:z1-condition} as
  \begin{equation}
    z_1^2=\frac{J^+(v_1)+J^+(u_{g-1})}{J^+(u_1)+J^+(v_{g-1})}.
  \end{equation}
  so that
  \begin{equation}
    \label{equation:3.5}
    \frac{J^+(v_1)+J^+(v_2)}{J^+(u_1)+J^+(u_2)}
    =
    \frac{J^+(v_1)+J^+(u_{g-1})}{J^+(u_1)+J^+(v_{g-1})}.
  \end{equation}

  In case the parity is \emph{odd}, then either~$J^+(u_2)=J^+(v_{g-1})$ and~$J^+(v_2)\neq J^+(u_{g-1})$ or vice versa.
  But the equality of \eqref{equation:3.5} forces both equalities in case one equality holds, so this situation does not arise.

  In the case the parity is \emph{even} we have listed the values of~$J^+$ in \cref{table:sign-choices-2}.
  In the first four cases we observe that~$J^+(v_2)=J^+(v_{g-1})=\pm2$ and~$J^+(v_2)=J^+(u_{g-1})=\pm2$ for the appropriate choice of signs.
  This implies that
  \begin{equation}
    \begin{aligned}
      z_1^2&=\frac{J^+(u_1)\pm2}{J^+(v_1)\pm2} \\
      z_2^2&=\frac{J^+(v_1)\pm2}{J^+(u_1)\pm2}
    \end{aligned}
  \end{equation}
  hence the choice of~$u_1$ and~$v_1$ determine~$z_1$ and~$z_2$ up to a sign.
  The equations become those for genus~$g-2$,
  with the additional constraint that~$u_{g-1}^2=v_{g-1}^2=1$.
  Now projecting along~$u_1$ and~$v_1$ reduces the equations,
  and we get that the dimension of the critical locus indexed by~$k$ (and genus~$g$)
  is two less than that of the critical locus indexed by~$k-2$ and genus~$g-2$,
  hence of expected dimension.

  It remains to study the last four cases.
  Now the induction will not go down by two, but rather by one.
  Hence we need to ensure the factor~$\sqrt{-1}$ comes into play.

  We claim that for $(\begin{smallmatrix} + & + \\ - & - \end{smallmatrix})$ and $(\begin{smallmatrix} - & + \\ - & + \end{smallmatrix})$,
  resp.~$(\begin{smallmatrix} - & - \\ + & + \end{smallmatrix})$ and $(\begin{smallmatrix} + & - \\ + & - \end{smallmatrix})$ we have that
  \begin{enumerate}
    \item $J^+(u_1)=-J^+(v_1)$ and $z_1^2=z_2^2=-1$;
    \item $J^+(u_1)=J^+(v_1)$ and $z_1^2=z_2^2=1$.
  \end{enumerate}
  We explain this in the case of $(\begin{smallmatrix} + & + \\ - & - \end{smallmatrix})$,
  the other cases being analogous.
  From \cref{table:sign-choices-2} and \eqref{equation:3.5} we obtain
  \begin{equation}
    \frac{J^+(v_1)-2}{J^+(u_1)+2}=\frac{J^+(v_1)+2}{J^+(u_1)-2}.
  \end{equation}
  This can be rewritten as
  \begin{equation}
    \frac{J^+(u_1)+J^+(v_1)}{J^+(u_1)+2}=\frac{J^+(u_1)+J^+(v_1)}{J^+(u_1)-2}.
  \end{equation}
  This is only possible if~$J^+(u_1)=-J^+(v_1)$.
  But we have that~$J^+(u_1)\neq\pm2$,
  so \eqref{equation:z1-condition} implies that~$z_1^2=-1$,
  and similarly \eqref{equation:zi-condition} for~$i=2$ implies that~$z_2^2=-1$.
  We will assume~$z_1=z_2^{-1}=\sqrt{-1}$, the other cases are similar.

  By the sign choices we have~$J^+(u_2)=J^+(u_{g-1})=2$, $J^+(u_{g-1})=J^+(v_2)=-2$.
  We will now use the expression of the graph potential as a sum of bead potentials,
  together with the partial evaluation we can do using the equalities from above.
  % long tedious computation omited
  Omitting a tedious manipulation,
  the presence of~$z_1=z_2=\sqrt{-1}$ will ensure that the resulting Laurent polynomial
  is the graph potential for the necklace graph of genus~$g-1$,
  except that it is multiplied with~$\sqrt{-1}$,
  so that all critical values are rotated by~90 degrees.
\end{proof}

\begin{table}
  \centering
  \begin{tabular}{c|cc|cc}
    & $J^+(u_2)$ & $J^+(v_{g-1})$ & $J^+(v_2)$ & $J^+(u_{g-1})$ \\
    \midrule % same order as in Swarnava's note: first (1)-(4) and then (13)-(16)
    $(\begin{smallmatrix} + & + \\ + & + \end{smallmatrix})$ & $2$ & $2$ & $2$ & $2$ \\
    $(\begin{smallmatrix} - & - \\ - & - \end{smallmatrix})$ & $-2$ & $-2$ & $-2$ & $-2$ \\
    $(\begin{smallmatrix} + & - \\ - & + \end{smallmatrix})$ & $2$ & $2$ & $-2$ & $-2$ \\
    $(\begin{smallmatrix} - & + \\ + & - \end{smallmatrix})$ & $-2$ & $-2$ & $2$ & $2$ \\
    $(\begin{smallmatrix} + & + \\ - & - \end{smallmatrix})$ & $2$ & $-2$ & $-2$ & $2$ \\
    $(\begin{smallmatrix} - & - \\ + & + \end{smallmatrix})$ & $-2$ & $2$ & $2$ & $-2$ \\
    $(\begin{smallmatrix} + & - \\ + & - \end{smallmatrix})$ & $2$ & $-2$ & $2$ & $-2$ \\
    $(\begin{smallmatrix} - & + \\ - & + \end{smallmatrix})$ & $-2$ & $2$ & $-2$ & $2$ \\
  \end{tabular}
  \caption{Values of $J^+$}
  \label{table:sign-choices-2}
\end{table}